\documentclass[11pt, onecolumn]{elsarticle}

\usepackage[margin=1in]{geometry}

\usepackage{amsmath}

\usepackage{setspace}

\usepackage{amsfonts}
\usepackage{amssymb}

\usepackage{amsthm}
\theoremstyle{definition}
\newtheorem{definition}{Definition}
\theoremstyle{plain}
\newtheorem{theorem}{Theorem}
\theoremstyle{plain}
\newtheorem{corollary}{Corollary}
\theoremstyle{plain}
\newtheorem{lemma}{Lemma}

\usepackage{comment}

\usepackage{lipsum}
\usepackage{graphicx}

\usepackage{multirow}

\usepackage{tikz}
\usetikzlibrary{arrows,decorations.markings}
\usetikzlibrary{shapes.geometric, arrows}

\usepackage{ctable}
\usepackage{algorithm}
\usepackage{algpseudocode}
\algrenewcommand\algorithmicrequire{\textbf{Input:}}
\algrenewcommand\algorithmicensure{\textbf{Output:}}

\tikzstyle{startstop} = [rectangle, rounded corners, 
minimum width=3cm, 
minimum height=1cm,
text centered, 
draw=black, 
fill=red!30]

\tikzstyle{io} = [trapezium, 
trapezium stretches=true, 
trapezium left angle=70, 
trapezium right angle=110,  
minimum height=1cm, text centered, 
draw=black, fill=blue!30]

\tikzstyle{process} = [rectangle,  
minimum height=1cm, 
text centered, 
draw=black, 
fill=orange!30]

\tikzstyle{decision} = [diamond,  
minimum height=0cm, 
text centered, 
draw=black, 
fill=green!30]
\tikzstyle{arrow} = [thick,->,>=stealth]

\tikzstyle{process2} = [
    process,
    fill=blue!30
]

\newtheorem*{theorem-non}{Theorem}

\usepackage[colorlinks=true,breaklinks=true,bookmarks=false,urlcolor=blue,citecolor=blue,linkcolor=blue,bookmarksopen=false,draft=false]{hyperref}
\usepackage{doi}

\begin{document}

\begin{frontmatter}

\title{A Canceling Heuristic for the Directed Traveling Salesman Problem}

\author{Steffen Borgwardt\corref{cor1}\fnref{label1}} \ead{steffen.borgwardt@ucdenver.edu}
\cortext[cor1]{Corresponding author}

\affiliation[label1]{organization={Department of Mathematical and Statistical Sciences, University of Colorado Denver},
            addressline={1201 Larimer Street}, 
            city={Denver},
            postcode={80204}, 
            state={CO},
            country={USA}}

            \author{Zachary Sorenson\fnref{label1}} 
            \ead{zachary.sorenson@ucdenver.edu}

\begin{abstract}
The Traveling Salesman Problem (TSP) is one of the classic and hard problems in combinatorial optimization. We develop a new heuristic that uses a connection between Minimum Cost Flow Problems and the TSP to improve on a given suboptimal tour, such as a local optimum found using a classic heuristic.

Minimum Cost Flow Problems can be solved efficiently through linear programming or combinatorial algorithms based on cycle canceling. We investigate the potential of flow-canceling in the context of the TSP. Through a restriction of the search space to cycles and circulations that alternate between arcs in- and outside of the tour, practical results exhibit that only a low number of subtours is created, and a lightweight patching step suffices for a high success rate and gap closure towards an optimum.
\end{abstract}

\begin{keyword}Traveling Salesman \sep Network Flows \sep Heuristic \sep Linear Programming

\MSC[2020] 90C05 \sep  90C27 \sep  90C35 \sep  90C59
\end{keyword}

\end{frontmatter}

\section{Introduction}

The {\em Traveling Salesman Problem (TSP)} is one of the fundamental problems in combinatorial optimization and integer programming \cite{abcc-07,pcss-24}. A standard description of the classic TSP is as follows \cite{Im-14}: A traveling salesman must visit $n$ cities and return to an initial city. The route that he travels is called a {\em tour}. Each time he travels from some city $i$ to some city $j$, it comes at a cost $c_{ij}>0$. To solve this problem, one must find a tour that minimizes the total cost of traveling to every city (and returning to the first city). There are several variations of the problem. For example, if there exists a pair of cities $i$ and $j$ in a TSP where $c_{ij} \neq c_{ji}$, then the problem is referred to as the {\em directed} or {\em asymmetric TSP}. If $c_{ij} = c_{ji}$ for every pair of cities $i$ and $j$, the problem is known as a {\em symmetric TSP}. Other common variants of the TSP include the {\em Euclidean TSP}, for which the distances between cities correspond to distances in the Euclidean plane. In most settings, distances for the TSP are assumed to at least satisfy the triangle inequality. This assumption leads to an optimal tour avoiding any repeat of a city. In this work, we do \emph{not} assume satisfaction of the triangle inequality, but we search for a tour without repetition of cities, as is standard.

The TSP is readily represented on a graph $G=(V,E,c)$ with a vertex set $V$, an edge set $E$, and edge costs $c$. The goal is the search for a Hamiltonian cycle of minimum cost. Symmetric TSPs are represented on undirected graphs (using undirected edges), while directed TSPs are represented on directed graphs (using directed arcs). For all variants, it is safe to assume a complete, loopfree graph and penalize undesired edges sufficiently so that they cannot appear in an optimal solution. 

In this work, we are interested in the directed TSP and devise a new heuristic for it based on a network flows approach. In Section \ref{sec:relatedwork}, we recall some necessary background from the literature. In Section \ref{sec:contributions}, we explain our contributions and provide an outline for the remainder of the paper. 

\subsection{Background}\label{sec:relatedwork}

It is well known that, unless $\mathcal{P}=\mathcal{NP}$, there cannot exist an exact efficient algorithm to solve the TSP. All common variants of TSP are known to be $\mathcal{NP}$-Hard; see e.g., \cite{B-71,Gar-79,C-77,W-87}. A naive approach to solve the problem exactly is to check every possible order of locations at an exponential running time of $O(n!)$. Alternatively, there exist dynamic programming algorithms that achieve a more efficient running time of $O(n^2 2^n)$ for such a check \cite{B-76}. 

In practice, local optima for the directed TSP are found using a variety of heuristics \cite{rggo-11}. Two classic heuristics, the $k$-opt algorithm and the patching algorithm, are relevant to our work. 

The {\em $k$-opt algorithm} \cite{J-95} is an iterative, local search algorithm. An initial tour and the number $k$ are chosen at the start. In each iteration, one selects and removes $k$ edges from the current tour $T$ to obtain $k$ segments of $T$. Then, one tests all the possible orderings of the segments to find a better tour. In graph-theoretical terms, each of these orderings corresponds to a reconnection of the segments at their endpoints through $k$ new edges. This process is repeated until no selection of $k$ edges can improve the current tour anymore. The widely-used {\em Lin-Kernighan algorithm} \cite{Keld,K-73} generalizes this idea to not fix $k$ at the start, but instead increase it dynamically if no improvement is possible for the current $k$; upon an improvement, $k$ is reset. $k$-opt and Lin-Kernighan are well-performing and popular heuristics for the symmetric TSP.  For directed TSP, they also are a common approach, but their performance is impacted by the fact that the reconnection of segments may lead to the reversal of the tour on one or more segments. The costs of the segment and reversed segment are not necessarily related. 

The {\em patching algorithm} \cite{Patch-78} first computes a cheapest cover of $V$ by a set of directed cycles. This leads to a partition of $V$ into vertex-disjoint directed cycles on proper subsets of $V$ such that each vertex lies in some cycle. As long as there exist two or more cycles, the algorithm performs a {\em patching operation}, or a \emph{patch}, of two cycles: The simultaneous deletion of an arc from each cycle and the addition of two new arcs, such that the total number of cycles in the partition is reduced by one. This process is repeated until a tour is formed. Unlike $k$-opt, the patching algorithm is not an iterative algorithm.

Mathematical programming-based approaches to the directed TSP build on one of multiple formulations as an integer program. One such formulation is as follows \cite{P-03}:

\vspace*{-0.5cm}

\begin{align*}\label{TSP-IP}\tag{TSP-IP}
    \text{minimize} \hspace{2mm} \sum_{i\in V} \sum _{j\in V\backslash\{i\}} x_{ij} \cdot   c_{ij} & \\
    \sum_{i \in S}\sum_{j \in S\backslash\{i\}} x_{ij} & \leq |S|-1 &  \forall\ S \subsetneq V, |S| \geq 2 \\
    \sum_{i \in V\backslash\{j\}}  x_{ij} & = 1  &  \forall\ j \in V \\
    \sum_{j \in V\backslash\{i\}}  x_{ij} & = 1  &  \forall\ i \in V \\
      x_{ij} & \in \{0,1\} & \forall\ i \in V, j \in V\backslash\{i\}
\end{align*}

Decision variables $x_{ij} \in \{0,1\}$ represent whether arc $(i,j)$ is part of the tour $T=x$. The objective function is the total sum of the costs of each arc used.  The second and third sets of constraints ensure that every vertex in a solution has one ingoing and one outgoing arc. The first set of constraints, known as the {\em subtour elimination constraints}, prevents the formation of {\em subtours}, i.e., directed cycles on proper subsets of $V$, by checking that each proper subset of vertices in the graph does not include enough arcs to form one. Note that IP (\ref{TSP-IP}) is of exponential size due to the specification of a subtour elimination constraint for each proper subset $S$ of $V$ of size $|S|\geq 2$. In practice, classic tools of integer programming -- such as branch-and-cut and decomposition models \cite{shwl-26} -- are used in combination with powerful metaheuristics to approach an exact or approximate solution without the need for a complete specification. This remains a highly active field of study; see \cite{Im-14} and references therein, as well as the seminal work \cite{APP} and related Concorde solver. Heuristics are not only the standard approach when a locally optimal solution is sufficient, but also can be a key subroutine for approximate or exact solution.

In this work, we are interested in repurposing tools commonly used for network flows problems for tour improvement. To this end, a key observation is that IP~(\ref{TSP-IP}) can be viewed as a restriction of a $0,1$-circulation LP. Circulation problems are special minimum cost flow problems with flow balance at each vertex. We here use unit capacities on arcs; a $0,1$-circulation is restricted to unit flow, i.e., flow $0$ or $1$ on each arc. A formulation is as follows \cite{amo-93}:

\vspace*{-0.5cm}

\begin{align*}\label{Circ}\tag{Circulation}
    \text{minimize} \hspace{2mm} \sum_{i\in V} \sum _{j\in V\backslash\{i\}} x_{ij} \cdot   c_{ij} & \\
    \sum_{j \in V\backslash\{i\}}  x_{ij} -\sum_{j \in V\backslash\{i\}}  x_{ji} & = 0  &  \forall\ i \in V \\
     0 \leq x_{ij} & \leq 1 & \forall\ i \in V,j \in V\backslash\{i\}
\end{align*}
The first set of constraints specifies balance of inflow and outflow at each vertex. They form a totally-unimodular constraint matrix, which implies that for integral right-hand sides all vertices of the underlying polytope are integral. Thus, it suffices to specify $0 \leq x_{ij} \leq 1$ instead of $x_{ij} \in \{0,1\}$. Further, all feasible $0,1$-circulations are vertices.

Any tour in $G$, i.e., any feasible solution to IP (\ref{TSP-IP}), is a special, spanning $0,1$-circulation (no subtours are formed and each vertex has inflow and outflow one)  and thus corresponds to a vertex of LP (\ref{Circ}). Formally, the difference between IP~(\ref{TSP-IP}) and LP~(\ref{Circ}) is the drop of subtour elimination constraints and the summation (or difference) of the constraints for ingoing and outgoing arcs. The objective function remains the same. 

This modeling link lets us import cycle‑canceling machinery from classic min‑cost flow algorithms, which build a residual graph $R=(G,x)$ from the graph $G$ and current flow $x$, search for a negative cycle $C$ in $R$, and then \emph{cancel} this cycle, i.e., send flow along it, to obtain a new flow of lower cost \cite{amo-93}. Specializing to a flow $x$ that corresponds to a tour $T$ and arcs with unit capacities, 
we work with a special residual network $R=(G,T)$, which we call a \emph{tour graph}; we provide a formal definition in Section \ref{sec:cancel}, Definition \ref{def:tourgraph}. At the heart of this work is a \emph{cancel step} that sends unit flow along carefully structured cycles or circulations in the tour graph. 

\subsection{Contributions and Outline}\label{sec:contributions}
We introduce \textsc{Cycap}, for ``cycle cancel and patch'', a heuristic for the directed Traveling Salesman Problem (and also applicable to the symmetric TSP). The method repurposes well–known ideas from minimum–cost network flows -- residual networks and cycle canceling -- to improve a given tour. At a high level, \textsc{Cycap} (i) constructs a residual network based on the current tour, (ii) detects a negative ``tour-alternating'' cycle or a circulation in a carefully chosen search space, (iii) performs a unit–flow \emph{cancel step} along that structure to reduce cost, and (iv) patches any subtours resulting from the cancel step back together. 

Our aim is to offer a routine to escape local optima found by classic local search heuristics (such as $2$–opt or $3$–opt) at practical running times. For this reason, the steps of \textsc{Cycap} are embedded in a larger computational process. A visual representation of such a complete process, where \textsc{Cycap} is combined with a $k$-opt algorithm, is provided as a flow chart in Figure \ref{fig:flow}. The red ovals indicate start and end of the process; it begins with an input graph and ends with a final tour. The blue, double-bordered rectangles and brown, single-bordered rectangle represent subroutines of the heuristic. Blue double–bordered rectangles denote subroutines drawn and tailored from the literature. Green diamonds are decision points (automated checks or user choices). We highlight the {\em Separation of Tour Graph} step for two reasons: It is the entry point into the core of \textsc{Cycap} embedded in this larger process, and also a key aspect in the design of our algorithm.

 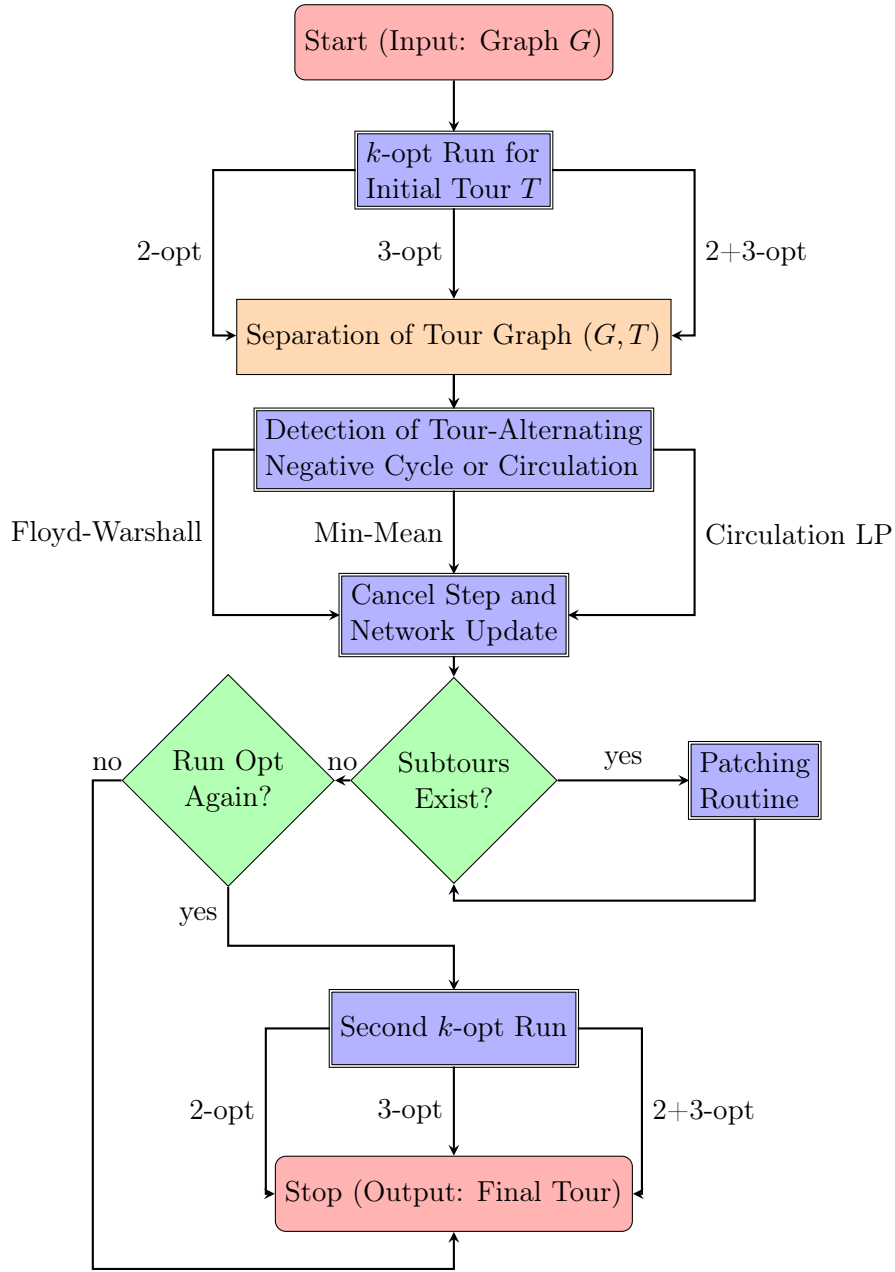
\begin{figure}

 \begin{center}
 \begin{tikzpicture}[node distance=1.5cm]

 \node (start) [startstop] {Start (Input: Graph $G$)};

 \node (first1) [process2, below of=start, yshift=-0.2cm, align=left,double] {$k$-opt Run for\\ Initial Tour $T$};

 \node (proc2a) [process, below of=first1, yshift=-0.72cm, align=left] {Separation of Tour Graph $(G,T)$};

 \node (out1) [process2, double, below of=proc2a, yshift=-0cm, align=left] {Detection of Tour-Alternating\\ Negative Cycle or Circulation};

 \node (pro3a) [process2, double, below of=out1, yshift=-0.7cm, xshift=0cm,align=center] {Cancel Step and\\ Network Update};

 \node (deci1) [decision, below of=pro3a,align=center,yshift=-0.7cm] {Subtours\\Exist?};

 \node (pro4a) [process2, double, right of=deci1, xshift=2.5cm,align=left] {Patching\\Routine};

 \node (decia1) [decision, left of=deci1, xshift=-1.5cm,align=center] {Run Opt\\Again?};

 \node (last1) [process2,,double,below of=deci1, yshift=-1.8cm,align=center] {Second $k$-opt Run};

 \node (stop) [startstop,below of=deci1, yshift=-4cm,align=center] {Stop (Output: Final Tour)};

 \draw [arrow] (start) -- (first1);
 \draw [arrow] (first1) -- node[anchor=east] {3-opt} (proc2a);
 \draw [arrow] (first1) --  ++(-3.2,0) -- node[anchor=east,yshift=-0.45] {2-opt} ++(0,-2.2) -- (proc2a);
 \draw [arrow] (first1) --  ++(3.2,0) -- node[anchor=west,yshift=-0.45] {2+3-opt} ++(0,-2.2) -- (proc2a);
 \draw [arrow] (proc2a) -- (out1);
 \draw [arrow] (out1) -- node[anchor=east] {Min-Mean} (pro3a);
 \draw [arrow] (out1) --  ++(-3.2,0) -- node[anchor=east,yshift=-0.45] {Floyd-Warshall} ++(0,-2.2) -- (pro3a);
 \draw [arrow] (out1) --  ++(3.2,0) -- node[anchor=west,yshift=-0.45] {Circulation LP} ++(0,-2.2) -- (pro3a);
 \draw [arrow] (proc2a) -- (out1);
 \draw [arrow] (pro3a) -- (deci1);
 \draw [arrow] (deci1) -- node[anchor=south] {yes} (pro4a);
 \draw [arrow] (pro4a) -- ++(0,-1.6) -- ++(-4,0) -- (deci1);
 \draw [arrow] (deci1) -- node[anchor=south] {no}  (decia1); 
 \draw [arrow] (decia1) --node[anchor=south] {no} ++(-1.8,0) -- ++(0,-6.5) -- ++(4.8,0) --  (stop);
 \draw [arrow] (decia1) --node[anchor=east] {yes} ++(0,-2.2) -- ++(3,0) --  (last1);
 \draw [arrow] (last1) -- node[anchor=east] {3-opt} (stop);
 \draw [arrow] (last1) --  ++(-2.5,0) -- node[anchor=east,yshift=-0.45] {2-opt} ++(0,-2.2) -- (stop);
 \draw [arrow] (last1) --  ++(2.5,0) -- node[anchor=west,yshift=-0.45] {2+3-opt} ++(0,-2.2) -- (stop);

 \end{tikzpicture}

 \caption{Flow Chart for the \textsc{Cycap} Heuristic}\label{fig:flow}

 \end{center}
 \end{figure}

Let us provide a brief overview of the computational process visualized in Figure~\ref{fig:flow}.  Beginning with an input graph $G$, a baseline $k$-opt run produces a locally optimal tour that serves as the starting point (\emph{“$k$–opt Run for Initial Tour $T$”}). The indicated variants are $2$-opt, $3$-opt, and $2+3$-opt, which refers to using $3$-opt when $2$-opt gets stuck. The core of \textsc{Cycap} then begins via the \emph{Separation of Tour Graph $(G,T)$}: We build the residual \emph{tour graph} $R=(G,T)$ and transform it into a special bipartite graph $B$, which we will call a \emph{separated graph}; this transformation guarantees that cycles detected in $B$ map back to \emph{tour–alternating} cycles in $R$ with identical cost. Next comes the \emph{Detection of Tour–Alternating Negative Cycle or Circulation}, which is performed through a classic negative cycle or circulation detection routine on $B$. The three indicated variants are the Floyd–Warshall algorithm with an adjusted predecessor readout; the search of a minimum-mean cycle; or the search of a minimum–cost circulation via a linear program.  The selected structure is then translated back to $R$ and applied in the \emph{Cancel Step and Network Update}: We send unit flow along the cycle/circulation in $R$ to reduce cost. This update may result in a collection of subtours and possibly some isolated vertices. At the first decision point (\emph{Subtours Exist?}), we decide if a \emph{Patching Routine} is required to reconnect them into a single tour. A second decision point (\emph{Run Opt Again?}) optionally triggers a \emph{Second $k$–opt Run} starting from the current tour, after which the process terminates with the final tour. In this work, we present some structural insight, graph constructions, algorithms, and computational evidence related to this pipeline. 

In Section \ref{sec:cancel}, we first define the \emph{tour graph} $R=(G,T)$ (Definition  \ref{def:tourgraph}) and
\emph{tour–alternating cycles and circulations} (Definitions \ref{def:ta-cycle} and \ref{def:ta-circ}) -- structures in $R$ whose arcs alternate between reversing a tour arc and inserting a non–tour arc. Section \ref{sec:searchspace} is then dedicated to an explanation and motivation of these tour-alternating cycles and circulations as a search space for cancellation on the current tour. We begin with an example to illustrate the concept and its potential to leave a local optimum for $k$-opt, and proceed into some formal observations. First,  we show that any simple cycle whose cancel produces a new tour must be tour–alternating, and that canceling a tour–alternating cycle returns either a tour or a set of vertex‑disjoint subtours (Theorem~\ref{thm:single-cycle}). Further, whenever a non‑simple cycle or a circulation yields a tour, it decomposes into a collection of arc‑disjoint tour–alternating cycles (Theorem~\ref{thm:alt-decomp}). Most of the arguments for the cancellation of a cycle transfer to circulations, but become more technical: A circulation may pick up opposite non-tour arcs in the tour graph, which we choose to trim. The result is the creation of a tour, a collection of subtours, or subtours together with some isolated vertices (Corollary~\ref{cor:single-cycle}). The key takeaway is that, regardless of the variant, a cancel step along a tour–alternating cycle or circulation remains in the proximity of a tour, and a lightweight patching routine suffices to return to a tour.

In Section \ref{sec:compstrategy}, we lay out an efficient computational strategy to find tour-alternating structures. We define the bipartite \emph{separated graph} $B$ (Definition \ref{def:separated}) obtained from the tour graph $R=(G,T)$ by duplicating the vertex set and routing insertion and removal arcs in opposite directions. In this construction, we retain a cycle/cost correspondence (Theorem \ref{thm:sepgraph-correspondence}): Every directed cycle in $B$ bijectively maps to a tour–alternating cycle in $R$ with identical total cost, and the same statements extend cyclewise to circulations. A simple and efficient matrix–based construction of $B$ (Algorithm \ref{alg:transf})  builds the distance and predecessor matrices in running time $\Theta(|V|^2)$  (Lemma~\ref{lem:sep-construct}). 

In Section \ref{sec:implementation}, we provide some implementation details for practical realizations of \textsc{Cycap} and the supporting subroutines. In Section \ref{sec:variants}, we distinguish three variants we implemented that differ based on how they operate on the separated graph $B$: \textsc{Cycap–F}, where we run Floyd–Warshall to completion to gather a set of candidate cycles; \textsc{Cycap–M}, where we search for a minimum–mean cycle; and \textsc{Cycap–C}, where we solve a minimum–cost circulation problem. In Section \ref{sec:tailoredclassic}, we describe the tailoring of three classic algorithms that enable the variants of \textsc{Cycap}. First, we use an adjusted $k^{\ast}$–opt algorithm for directed instances that explicitly accounts for segment reversals (Algorithm~\ref{alg:kstar-opt}), and thus gives better local optima. Second, we use an adjusted predecessor readout, a variant of Floyd-Warshall predecessor tracking designed to harvest all negative cycles found in a complete run (Algorithm~\ref{alg:adjreadout}). This routine is used in \textsc{Cycap-F}. And finally, we use an adjusted patching routine that merges subtours by cheapest two‑arc patches and absorbs isolated vertices by a three‑arc rule.

In Section \ref{sec:experiments}, we present our computational experiments on benchmark instances from the widely-used TSPLIB repository \cite{TLIB}. We describe experiment setup and protocol, and then report on three metrics: \emph{success rate} (improvement after \textsc{Cycap} alone and after \textsc{Cycap} plus a second $k$–opt run), \emph{gap closure} toward the optimal or best known tour, and \emph{running time} (both \textsc{Cycap}–only and the overall process). We observe that on directed instances \textsc{Cycap–C} consistently achieves very high success rates both immediately after applying \textsc{Cycap} and after the optional second $k$–opt run, closes a large fraction of the gap to an optimum (often exceeding $70\%$), and runs faster than the cycle–based variants. In fact, \textsc{Cycap–C} is substantially faster than \textsc{Cycap–F} and \textsc{Cycap–M} on larger graphs -- often by an order of magnitude -- and the optional second $k$–opt run is typically brief because \textsc{Cycap–C} itself already closes much of the gap to an optimum. On symmetric instances, the cycle–based variants \textsc{Cycap–F} and \textsc{Cycap–M} perform competitively, as well, especially when followed by the additional $k$–opt run. In this setting, the three embedded variants have comparable total running times, while \textsc{Cycap–C} is the fastest when measured by \textsc{Cycap}–only time. 

We conclude in Section \ref{sec:conclusion} with some practical guidance and an outlook on promising directions of future work. A recurring empirical feature across our experiments is that cancel steps typically generate only a few subtours; this keeps the patching routine to minimal impact and is a key contributor to a strong overall performance. 
The creation of few subtours is linked to tour-alternating cycles and circulations having only few arcs. We share some basic observations about this link, and are interested in a deeper theoretical understanding of it.

\section{The Cancel Step}\label{sec:cancel}
This section develops the cancel step that sits at the heart of our heuristic. In Section~\ref{sec:searchspace} we formalize the \emph{search space} of \emph{tour–alternating cycles and circulations}, explain why this restriction guarantees that a cancel step produces either a single tour or a collection of subtours, and record basic structural properties needed later. In Section~\ref{sec:compstrategy} we outline the \emph{computational strategy}: how we detect such cycles or circulations efficiently via a \emph{separated graph} on which we apply classical routines, and how a unit–flow cancel translates back to the tour graph. We begin with some formal terminology and notation.

We work on a complete directed graph $G=(V,E,c)$ with vertex set $V=\{1,\dots,n\}$, arc set $E=V\times V\setminus\{(i,i):i\in V\}$, and nonnegative arc costs $c_{i,j}$ for $(i,j)\in E$.
A \emph{tour} $T$ in $G$ is a directed Hamiltonian cycle and is represented by a \emph{successor map} $\pi:V\to V$, a permutation with $\pi(i)\neq i$ for all $i$, so that the arc $(i,\pi(i))$ belongs to the current tour. Its cost is
\[
c(T)\;=\;\sum_{i\in V} c_{i,\pi(i)}\,.
\]
We denote by $R=(G,T)$ the \emph{tour graph}, i.e., a residual network built from $G$ with unit capacities and the unit flow induced by $T$. A formal definition is as follows.

\begin{definition}[Tour graph]\label{def:tourgraph}
Let $G=(V,E,c)$ be a complete directed graph with arc costs $c_{i,j}\ge 0$, and let $T$ be a tour given by its successor map $\pi:V\to V$. Define unit capacities $u_{i,j}=1$ 
on every arc, and the flow induced by $T$ as
\[
x_{i,j}\;=\;\begin{cases}
1, & j=\pi(i),\\[0.25ex]
0, & \text{otherwise}.
\end{cases}
\]
The \emph{tour graph} is $R=(G,T)=(V,E_R,c^R)$, where
\[
E_R \;=\; \{(h,k)\in E:\,(h,k),(k,h)\notin T\}\ \cup\ \{(\pi(i),\,i):\,i\in V\}.
\]
Residual costs are
\[
c^R_{h,k}=c_{h,k}\quad\text{for }(h,k),(k,h)\notin T,
\qquad
c^R_{\pi(i),\,i}=-\,c_{i,\pi(i)}\quad\text{for }(i,\pi(i))\in T,
\]
and all residual arcs have unit capacity.
\end{definition}

We will call the two types of arcs in $E_R$ \emph{insertion} arcs $E^{+}=\{(h,k)\in E:\,(h,k),(k,h)\notin T\}$, adding flow along non-tour arcs with cost $+c$, and \emph{removal} arcs $E^{-}=\{(\pi(i),\,i):\,i\in V\}$, removing flow from tour arcs via their reverse with cost $-c$. Note that $E^{+}$ consists of pairs $(h,k),(k,h)$ of opposite arcs (retaining their original costs) where neither of them is in $T$, whereas arcs $(i,\pi(i))$ do not appear in $E_R$.

The same notation for map, cost, and associated flow is used for \emph{subtours}, which are directed cycles supported on a proper subset of $V$. When we refer to a \emph{collection of subtours}, we assume that the collection is \emph{spanning} unless stated otherwise, i.e., all vertices in $V$ are part of some subtour. We also assume that subtours in a collection are vertex- and arc-disjoint. When a collection of subtours is not spanning, there exist one or more \emph{isolated vertices}, i.e., vertices with in-degree and out-degree zero in the current flow.

When discussing indices on cyclic examples, we write them modulo $n$ and use $0 = n$ for convenience.
For any set $S\subseteq V$, we write $G[S]$ for the induced subgraph and use $\pi(S)=\{\pi(i):i\in S\}$.

\subsection{The Search Space}\label{sec:searchspace}
Intuitively, a \emph{cancel step} removes certain tour arcs and inserts certain non–tour arcs. We focus on structures in $R$ that \emph{alternate} between these two roles. First, we define a \emph{tour–alternating cycle}.

\begin{definition}[Tour–alternating cycle]\label{def:ta-cycle}
Let $R=(G,T)$ be a tour graph. A directed cycle $A$ in $R$ is \emph{tour–alternating} if for every pair of consecutive arcs in $A$, exactly one is in $E^{-}$ and one is in $E^{+}$.
\end{definition}
Any arc sequence of $A$ alternates between $E^{-}$ and $E^{+}$.
Sending one unit of flow along $A$ decreases tour flow on the (reversed) removal arcs and increases flow on the insertion arcs, preserving in- and outflow at each visited vertex. Note that the definition does \emph{not} assume simplicity of the cycle $A$. Figure \ref{fig:floydunion} shows an example of a non-simple tour-alternating cycle.

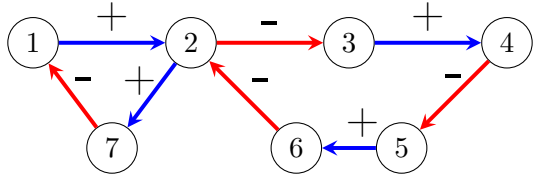
\begin{figure}[h]
\begin{center}
    \begin{tikzpicture}[shorten >=2pt, scale=0.7]
    \path (0,0) node[circle,draw] (x) {   $1$}
    (3,0) node[circle,draw](y) {   $2$}
    (6,0) node[circle,draw](z) {   $3$}
    (9,0) node[circle,draw](w) {   $4$}
    (7,-2) node[circle,draw](t) {   $5$}
    (5,-2) node[circle,draw](g) {   $6$}
    (1.5,-2) node[circle,draw](h) {   $7$};
    
    \draw[line width=1.5pt,blue, font=\bf, decoration={markings,mark=at position 1 with
    {\arrow[scale=1,>=stealth]{>}}}, postaction={decorate}]  (x) -- node[above,black] {\scalebox{1.3}[1.3]+} (y);
    \draw[line width=1.5pt,red, font=\bf, decoration={markings,mark=at position 1 with
    {\arrow[scale=1,>=stealth]{>}}}, postaction={decorate}]  (y) -- node[above,black] {\scalebox{1.8}[1.2]-} (z);
        \draw[line width=1.5pt,blue,font=\bf,  decoration={markings,mark=at position 1 with
    {\arrow[scale=1,>=stealth]{>}}}, postaction={decorate}]  (z) -- node[above,black] {\scalebox{1.3}[1.3]+} (w);
        \draw[line width=1.5pt,red, font=\bf, decoration={markings,mark=at position 1 with
    {\arrow[scale=1,>=stealth]{>}}}, postaction={decorate}]  (w) -- node[above,black] {\scalebox{1.8}[1.2]-} (t);
        \draw[line width=1.5pt,blue,font=\bf,  decoration={markings,mark=at position 1 with
    {\arrow[scale=1,>=stealth]{>}}}, postaction={decorate}]  (t) -- node[above,black,near start] {\scalebox{1.3}[1.3]+} (g);
        \draw[line width=1.5pt,red,font=\bf,  decoration={markings,mark=at position 1 with
    {\arrow[scale=1,>=stealth]{>}}}, postaction={decorate}]  (g) -- node[above=5pt,black, near start] {\scalebox{1.8}[1.2]-} (y);
        \draw[line width=1.5pt,blue,font=\bf,  decoration={markings,mark=at position 1 with
    {\arrow[scale=1,>=stealth]{>}}}, postaction={decorate}]  (y) -- node[above=2pt,black,near end] {\scalebox{1.3}[1.3]+} (h);
        \draw[line width=1.5pt,red,font=\bf,  decoration={markings,mark=at position 1 with
    {\arrow[scale=1,>=stealth]{>}}}, postaction={decorate}]  (h) -- node[above=5pt,black, near start] {\scalebox{1.8}[1.2]-} (x);
 
    \end{tikzpicture}
    
\end{center}

\caption{A non-simple tour-alternating cycle. Insertion arcs are blue, removal arcs are red.}\label{fig:floydunion}
\end{figure}

A \emph{tour–alternating circulation} $C$ is a union of tour–alternating cycles; canceling $C$ sends one unit of flow along each constituent cycle and preserves flow balance globally.

\begin{definition}[Tour–alternating circulation]\label{def:ta-circ}
Let $R=(G,T)$ be a tour graph. A \emph{circulation} $C$ in $R$ is \emph{tour–alternating} if it can be decomposed into tour-alternating cycles. 
\end{definition}

We begin with an illustrative example that showcases a cancel step and motivates tour–alternating structures. Consider the network and tour shown in Figure \ref{fig:10tour}a). It has $10$ vertices $1,\dots,10$. Let $T$ be a tour whose arcs alternate between blue arcs of cost $2$ and black arcs of cost $12$, so that 
\[
c(T)\;=\;5\cdot 2\;+\;5\cdot 12\;=\;70.
\]
Assume the non–tour red arcs $(i,i+5)$ have cost $7$, and all remaining arcs have a prohibitive cost $C\gg 12$ so that any tour using them is strictly worse than $T$. We display opposite directed arc pairs $(i,j)$ and $(j,i)$ as a single undirected edge. 

 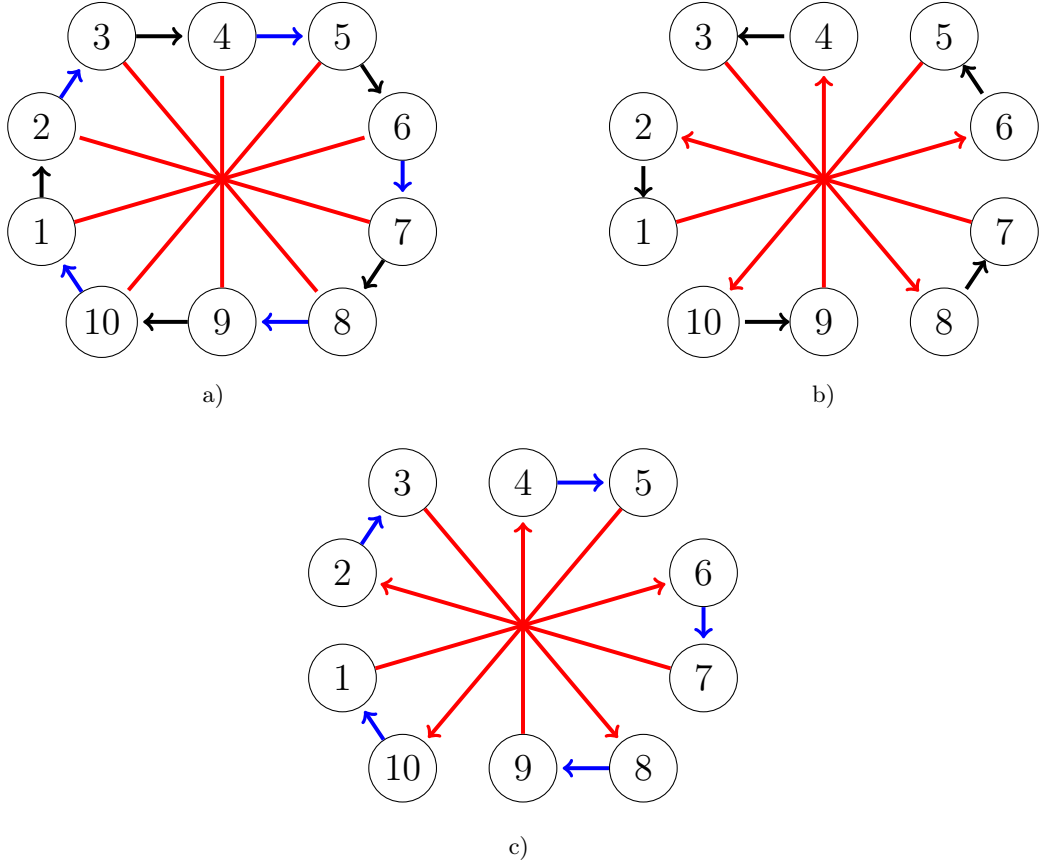
\begin{figure}[h]
\begin{center}
\begin{tikzpicture}[shorten >=2pt]
    \path (0,0) node[circle,draw,minimum size=9mm] (a) {\Large{$1$}}
    (0,1.4) node[circle,draw,minimum size=9mm](b) {\Large{$2$}}
    (0.8,2.6) node[circle,draw,minimum size=9mm](c) {\Large{$3$}}
    (2.4,2.6) node[circle,draw,minimum size=9mm](d) {\Large{$4$}}
    (4,2.6) node[circle,draw,minimum size=9mm](e) {\Large{$5$}}
    (4.8,1.4) node[circle,draw,minimum size=9mm](f) {\Large{$6$}}
    (4.8,0) node[circle,draw,minimum size=9mm](g) {\Large{$7$}}
    (4,-1.2) node[circle,draw,minimum size=9mm](h) {\Large{$8$}}
    (2.4,-1.2) node[circle,draw,minimum size=9mm](i) {\Large{$9$}}
    (0.8,-1.2) node[circle,draw,minimum size=9mm](j) {\Large{$10$}};
    \draw[->, black, line width = 1.5pt] (a) -- (b);
    \draw[->, blue, line width = 1.5pt] (b) -- (c);
    \draw[->, black, line width = 1.5pt] (c) -- (d);
    \draw[->, blue, line width = 1.5pt] (d) -- (e);
    \draw[->, black, line width = 1.5pt] (e) -- (f);
    \draw[->, blue, line width = 1.5pt] (f) -- (g);
    \draw[->, black, line width = 1.5pt] (g) -- (h);
    \draw[->, blue, line width = 1.5pt] (h) -- (i);
    \draw[->, black, line width = 1.5pt] (i) -- (j);
    \draw[->, blue, line width = 1.5pt] (j) -- (a);
    \draw[red, line width = 1.5] (a) -- (f);
    \draw[red, line width = 1.5] (c) -- (h);
    \draw[red, line width = 1.5] (i) -- (d);
    \draw[red, line width = 1.5] (g) -- (b);
    \draw[red, line width = 1.5] (e) -- (j);
    
    \path (8,0) node[circle,draw,minimum size=9mm] (a) {\Large{$1$}}
    (8,1.4) node[circle,draw,minimum size=9mm](b) {\Large{$2$}}
    (8.8,2.6) node[circle,draw,minimum size=9mm](c) {\Large{$3$}}
    (10.4,2.6) node[circle,draw,minimum size=9mm](d) {\Large{$4$}}
    (12,2.6) node[circle,draw,minimum size=9mm](e) {\Large{$5$}}
    (12.8,1.4) node[circle,draw,minimum size=9mm](f) {\Large{$6$}}
    (12.8,0) node[circle,draw,minimum size=9mm](g) {\Large{$7$}}
    (12,-1.2) node[circle,draw,minimum size=9mm](h) {\Large{$8$}}
    (10.4,-1.2) node[circle,draw,minimum size=9mm](i) {\Large{$9$}}
    (8.8,-1.2) node[circle,draw,minimum size=9mm](j) {\Large{$10$}};
    \draw[<-, black, line width = 1.5pt] (a) -- (b);
    \draw[<-, black, line width = 1.5pt] (c) -- (d);
    \draw[<-, black, line width = 1.5pt] (e) -- (f);
    \draw[<-, black, line width = 1.5pt] (g) -- (h);
    \draw[<-, black, line width = 1.5pt] (i) -- (j);
    \draw[->, red, line width = 1.5] (a) -- (f);
    \draw[->, red, line width = 1.5] (c) -- (h);
    \draw[->, red, line width = 1.5] (i) -- (d);
    \draw[->, red, line width = 1.5] (g) -- (b);
    \draw[->, red, line width = 1.5] (e) -- (j);
\end{tikzpicture}
    
    \vspace{0.2cm}
    \footnotesize{a)} \hspace{7.6cm} \footnotesize{b)}
    \vspace{0.5cm}

\begin{tikzpicture}[shorten >=2pt]
    \path (0,0) node[circle,draw,minimum size=9mm] (a) {\Large{$1$}}
    (0,1.4) node[circle,draw,minimum size=9mm](b) {\Large{$2$}}
    (0.8,2.6) node[circle,draw,minimum size=9mm](c) {\Large{$3$}}
    (2.4,2.6) node[circle,draw,minimum size=9mm](d) {\Large{$4$}}
    (4,2.6) node[circle,draw,minimum size=9mm](e) {\Large{$5$}}
    (4.8,1.4) node[circle,draw,minimum size=9mm](f) {\Large{$6$}}
    (4.8,0) node[circle,draw,minimum size=9mm](g) {\Large{$7$}}
    (4,-1.2) node[circle,draw,minimum size=9mm](h) {\Large{$8$}}
    (2.4,-1.2) node[circle,draw,minimum size=9mm](i) {\Large{$9$}}
    (0.8,-1.2) node[circle,draw,minimum size=9mm](j) {\Large{$10$}};
    \draw[->, blue, line width = 1.5pt] (b) -- (c);
    \draw[->, blue, line width = 1.5pt] (d) -- (e);
    \draw[->, blue, line width = 1.5pt] (f) -- (g);
    \draw[->, blue, line width = 1.5pt] (h) -- (i);
    \draw[->, blue, line width = 1.5pt] (j) -- (a);
    \draw[->, red, line width = 1.5] (a) -- (f);
    \draw[->, red, line width = 1.5] (c) -- (h);
    \draw[->, red, line width = 1.5] (i) -- (d);
    \draw[->, red, line width = 1.5] (g) -- (b);
    \draw[->, red, line width = 1.5] (e) -- (j);
    
\end{tikzpicture}

    \vspace{0.3cm}
    \footnotesize{c)}
\vspace*{-0.3cm}
\end{center}
    \caption{An example in which a cycle cancel outperforms $2$-opt and $3$-opt. a) Locally optimal tour for $2$-opt and $3$-opt b) Negative cycle in tour graph c) Improved tour formed by cancel}\label{fig:10tour}
\end{figure}

It is possible to improve on the tour through cancellation of a tour-alternating cycle. Figure \ref{fig:10tour}b) shows a tour-alternating cycle in the tour graph, alternating between black removal arcs and red insertion arcs. The cancel leads to the tour displayed in Figure \ref{fig:10tour}c), which retains the blue arcs from the original tour and replaces all black arcs with red ones. The cost of this improved tour $T'$ is \[
c(T')\;=\;5\cdot 2\;+\;5\cdot 7\;=\;45.
\]

There are a few noteworthy properties to this example. First, observe that we obtained a tour $T'$ directly from the cancel. Second, we are interested in canceling an improving tour-alternating cycle or circulation specifically when heuristics like $2$-opt or $3$-opt get stuck in a local optimum, and this is the case in this example:

Consider a possible update through $2$-opt. A $2$–opt move removes two tour arcs and reconnects to form a tour. Let $(i,i+1)$ be one of the removed arcs. Then avoiding a $C$–cost arc forces the second removed arc to be $(i+5,i+6)$. In this case, one blue and one black arc are replaced by two red arcs, and the net change is
\[
\Delta c\;=\;(2+12)\;-\;(7+7)\;=\;14-14\;=\;0,
\]
so no improvement is obtained.
All other choices pick up at least one $C$–cost arc and are strictly worse. Hence $T$ is locally optimal for $2$–opt.

Similarly, a $3$–opt move removes three tour arcs and reconnects to form a tour. Let $(i,i+1)$ again be one of the removed arcs. As before, avoiding a $C$–cost arc forces a second removal $(i+5,i+6)$. Recall that the set of red arcs consists of arcs $(j,j+5)$ for all $j$. Thus, any choice of the third removed arc necessarily introduces a $C$–cost arc in the reconnection, so no improving $3$–opt move exists. Thus $T$ is locally optimal for $3$–opt, as well.

In this example, the cancel step was able to directly improve a tour that is locally optimal (stuck) with respect to $2$-opt and $3$-opt. We note that, in general (and as we will prove), a cancellation of a tour-alternating structure gives a collection of subtours, and so our heuristic includes a patching routine to reconnect these. Further, the tour-alternating cycle in Figure \ref{fig:10tour}b) removes $5$ arcs and reconnects the $5$ remaining segments of the original tour; it is one of the candidates that would be considered in a run of $5$-opt (or $5^*$-opt; c.f., Section \ref{sec:tailoredclassic}). In practice, $k$-opt is not commonly run for values higher than $3$ due to the prohibitive $\Omega(n^k)$ running time. A cancel step does not restrict the number of arcs in a cycle or circulation to be canceled, and so provides access to a subset of the possible changes corresponding to a hypothetical run of $k$-opt for high values of $k$. Additionally, a cancel step also considers options that disconnect a tour into a collection of subtours, which then have to be patched together again; see Section \ref{sec:tailoredclassic} for a description of our simple patching routine. In summary, neither of the search spaces of tour-alternating cycles or $k$-opt is contained in one another. This gives some potential for exiting a local optimum for the popular $k$-opt algorithms through a cancel step. In Section \ref{sec:experiments}, we will confirm the success of this strategy in practice. 

Let us record some structural results on the cancel step. First, we combine the preceding observations into a formal statement that characterizes which single–cycle updates can produce a tour and what tour–alternation guarantees after a unit–flow cancel. Specifically, we prove that it is the only class of flow updates guaranteed to preserve local balance after a unit–flow cancel.

\begin{theorem}\label{thm:single-cycle}
Let $R=(G,T)$ be a tour graph and let $A$ be a directed cycle in $R$. Then
\begin{enumerate}
\item If the cancellation of a \emph{simple} cycle $A$ produces a tour $T'\neq T$, then $A$ is a tour–alternating cycle.
\item If $A$ is tour–alternating, then the cancellation of $A$ yields a feasible $0,1$-circulation that is either (i) a tour or (ii) a collection of subtours.
\end{enumerate}
\end{theorem}

\begin{proof}
Assume the cancellation of a \emph{simple} directed cycle $A$ in the tour graph $R=(G,T)$ produces a tour $T'\neq T$. Let $v$ be any vertex visited by $A$. Because $A$ is simple, exactly two arcs of $A$ are incident to $v$: one entering $v$ and one leaving $v$. In $R$, each arc is either a removal arc in $E^{-}$ (the reversal of a tour arc) or an insertion arc in $E^{+}$ (a non–tour arc).

If both incident arcs at $v$ are removal arcs, then both tour arcs at $v$ are deleted and none is added; $v$ becomes isolated, so the cancellation cannot yield a tour. If both incident arcs are insertion arcs, then the cancellation adds a new incoming non–tour arc and a new outgoing non–tour arc at $v$ without deleting the original tour arcs; $v$ would have in-degree $2$ and out-degree $2$, again contradicting that the result is a tour. Hence, at every visited vertex $v$, the two incident arcs of $A$ must comprise exactly one removal and exactly one insertion arc. This shows that consecutive arcs of $A$ alternate between $E^{-}$ and $E^{+}$, i.e., $A$ is a tour–alternating cycle, which proves the first statement.

Suppose now that $A$ is tour–alternating (and not necessarily simple). Let $x$ denote the $0,1$–flow of the tour $T$. The post–cancel flow $x'$ can be written as
\[
x' \;=\; x \;+\; \chi_{A\cap E^{+}} \;-\; \chi_{A\cap E^{-}},
\]
where $\chi_{E'}$ is the arc–incidence vector of arc set $E'$. Fix a vertex $v$ visited by $A$. Each visit of $A$ to $v$ consists of two consecutive arcs of $A$ incident to $v$, one entering and one leaving $v$; by tour–alternation, exactly one of these arcs is an insertion arc and one is a removal arcs. Flow balance at $v$ remains unchanged by this pair of arcs. Summing over all visits to $v$ yields zero net change in inflow and outflow at $v$, and vertices not visited by $A$ are unchanged. 

Hence $x'$ is balanced at every vertex, i.e., a feasible $0,1$–circulation, with inflow and outflow at each vertex equal to one. 
Such a $0,1$–circulation greedily decomposes into a spanning collection of vertex-disjoint directed cycles. Therefore the cancellation of $A$ produces either a single tour or a collection of subtours, as claimed.
\end{proof}

Note that the first part of Theorem \ref{thm:single-cycle} is stated for \emph{simple} cycles. For a non-simple cycle in $R$, per–vertex adjustments to flow balance may compensate across multiple visits and the cycle itself need not alternate. However, it is easy to see that, whenever the cancellation of a non–simple cycle $A$ produces a new tour, $A$ admits a decomposition into tour–alternating cycles, and thus it is a tour-alternating circulation; the same argument holds for a general $0,1$-circulation $A$.

\begin{theorem}\label{thm:alt-decomp}
Let $R=(G,T)$ be a tour graph and let $A$ be a directed non–simple cycle or a $0,1$-circulation in $R$ whose cancellation produces a tour $T'\neq T$. Then $A$ is a tour-alternating circulation.
\end{theorem}

\begin{proof}
Since cancellation of $A$ yields a tour $T'$, at each vertex $v$ visited by $A$ the number of insertion arcs of $A$ incident to $v$ equals the number of removal arcs of $A$ incident to $v$.
In $R$, each vertex $v$ is incident to exactly two removal arcs -- namely, $(\pi(v),v)$ and $(v,\pi^{-1}(v))$ -- so $A$ visits $v$ at most twice.
If $A$ visits $v$ once, it uses one removal arc and one insertion arc at $v$.
If $A$ visits $v$ twice, it uses both removal arcs together with two insertion arcs at $v$ -- one entering and one leaving $v$.

We pair incident arcs at $v$ so that alternation and orientation are respected. If $A$ visits $v$ once, pair the unique removal arc in $A$ at $v$ with the unique insertion arc in $A$ at $v$. If $A$ visits $v$ twice, pair the removal arc $(\pi(v),v)$ with the insertion arc leaving $v$ and pair the removal arc $(v,\pi^{-1}(v))$ with the insertion arc entering $v$. These pairings 
consume all arcs of $A$ incident to $v$.

Starting from any arc of $A$ and following the directed order, extend the walk by taking, at each intermediate vertex, the paired arc (which is of opposite type). Because each vertex supplies a unique paired continuation, this construction closes to a directed (possibly non-simple) cycle $A_i$. By construction, consecutive arcs in $A_i$ alternate between insertion arcs and removal arcs at every step. We remove its arcs and repeat the procedure. This greedy extraction yields a collection of arc–disjoint tour–alternating cycles $A_1,\dots,A_q$ whose union is the arc set of $A$. Hence $A$ is a tour–alternating circulation, as claimed. 
\end{proof}

In summary, Theorems~\ref{thm:single-cycle} and~\ref{thm:alt-decomp} ensure that only a cancel step using a tour-alternating cycle or circulation can result in a new tour. Additionally, cancellation along a tour-alternating cycle (simple or not) yields a tour or a collection of subtours, so that a patching step may be the only post–processing required. It remains to consider the outcome of a cancellation of a tour-alternating circulation.

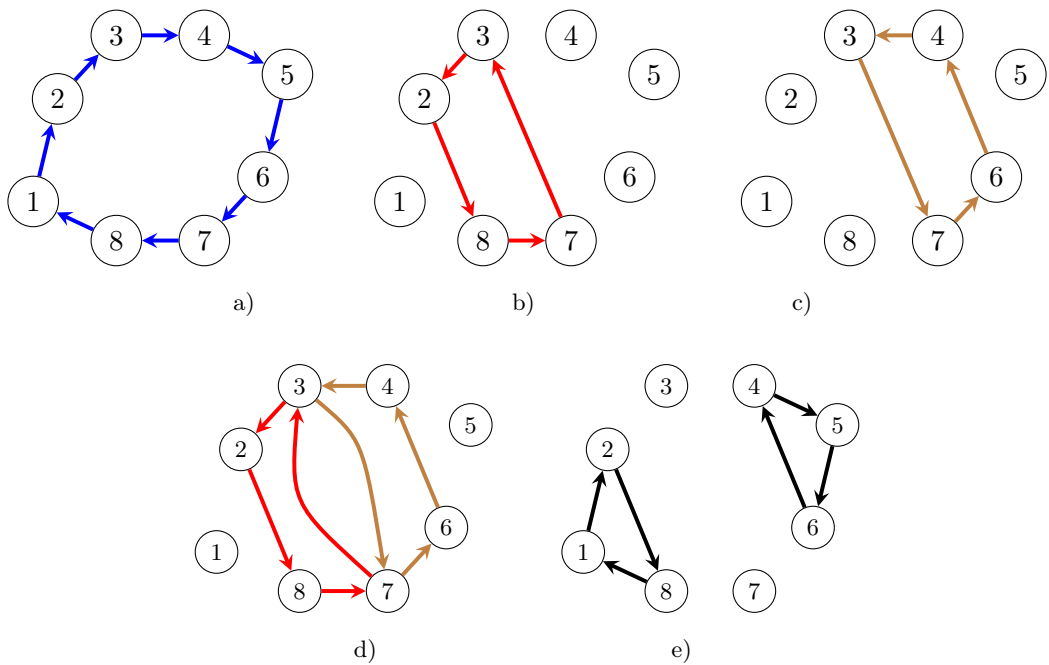
\begin{figure}[h]
\begin{center}
    \begin{tikzpicture}[shorten >=2pt, scale=0.65]
    \path (-0.5,-0.5) node[circle,,draw] (x) {   1}
    (0,1.6) node[circle,,draw] (k) {   2}
    (1.2,2.9) node[circle,draw](y) {    $3$}
    (1.2,-1.3) node[circle,draw](z) {    $8$}
    (3,2.9) node[circle,draw](w) {    $4$}
    (3,-1.3) node[circle,draw](t) {    $7$}
    (4.2,0) node[circle,draw](g) {    $6$}
    (4.7,2.1) node[circle,draw](m) {    $5$};
    
    \draw[line width=1.5pt,blue, decoration={markings,mark=at position 1 with
    {\arrow[scale=1,>=stealth]{>}}}, postaction={decorate}] (x) -- node[above,sloped] {} (k);
    \draw[line width=1.5pt,blue, decoration={markings,mark=at position 1 with
    {\arrow[scale=1,>=stealth]{>}}}, postaction={decorate}] (y) -- node[above,sloped] {} (w);
    \draw[line width=1.5pt,blue, decoration={markings,mark=at position 1 with
    {\arrow[scale=1,>=stealth]{>}}}, postaction={decorate}] (w) -- node[above,sloped] {} (m);
    \draw[line width=1.5pt,blue, decoration={markings,mark=at position 1 with
    {\arrow[scale=1,>=stealth]{>}}}, postaction={decorate}] (m) -- node[above,sloped] {} (g);
    \draw[line width=1.5pt,blue, decoration={markings,mark=at position 1 with
    {\arrow[scale=1,>=stealth]{>}}}, postaction={decorate}] (g) -- node[above,sloped] {} (t);
    \draw[line width=1.5pt,blue, decoration={markings,mark=at position 1 with
    {\arrow[scale=1,>=stealth]{>}}}, postaction={decorate}] (t) -- node[above,sloped] {} (z);
    \draw[line width=1.5pt,blue, decoration={markings,mark=at position 1 with
    {\arrow[scale=1,>=stealth]{>}}}, postaction={decorate}] (z) -- node[above,sloped] {} (x);
    \draw[line width=1.5pt,blue, decoration={markings,mark=at position 1 with
    {\arrow[scale=1,>=stealth]{>}}}, postaction={decorate}] (k) -- node[above,sloped] {} (y);
    
    \path (7,-0.5) node[circle,,draw] (x) {   $1$}
    (7.5,1.6) node[circle,,draw] (k) {   $2$}
    (8.7,2.9) node[circle,draw](y) {   $3$}
    (8.7,-1.3) node[circle,draw](z) {   $8$}
    (10.5,2.9) node[circle,draw](w) {   $4$}
    (10.5,-1.3) node[circle,draw](t) {   $7$}
    (11.7,0) node[circle,draw](g) {   $6$}
    (12.2,2.1) node[circle,draw](m) {   $5$};
    \draw[line width=1.5pt,red, decoration={markings,mark=at position 1 with
    {\arrow[scale=1,>=stealth]{>}}}, postaction={decorate}] (k) -- node[above,sloped] {} (z);
    \draw[line width=1.5pt,red, decoration={markings,mark=at position 1 with
    {\arrow[scale=1,>=stealth]{>}}}, postaction={decorate}] (y) -- node[above,sloped] {} (k);
    \draw[line width=1.5pt,red, decoration={markings,mark=at position 1 with
    {\arrow[scale=1,>=stealth]{>}}}, postaction={decorate}] (z) -- node[above,sloped] {} (t);
    \draw[line width=1.5pt,red, decoration={markings,mark=at position 1 with
    {\arrow[scale=1,>=stealth]{>}}}, postaction={decorate}] (t) -- node[above,sloped] {} (y);

    \path (14.5,-0.5) node[circle,,draw] (x) {   $1$}
    (15,1.6) node[circle,,draw] (k) {   $2$}
    (16.2,2.9) node[circle,draw](y) {   $3$}
    (16.2,-1.3) node[circle,draw](z) {   $8$}
    (18,2.9) node[circle,draw](w) {   $4$}
    (18,-1.3) node[circle,draw](t) {   $7$}
    (19.2,0) node[circle,draw](g) {   $6$}
    (19.7,2.1) node[circle,draw](m) {   $5$};

    \draw[line width=1.5pt,brown, decoration={markings,mark=at position 1 with
    {\arrow[scale=1,>=stealth]{>}}}, postaction={decorate}] (w) -- node[above,sloped] {} (y);
    \draw[line width=1.5pt,brown, decoration={markings,mark=at position 1 with
    {\arrow[scale=1,>=stealth]{>}}}, postaction={decorate}] (y) -- node[above,sloped] {} (t);
    \draw[line width=1.5pt,brown, 
    decoration={markings,mark=at position 1 with
    {\arrow[scale=1,>=stealth]{>}}}, postaction={decorate}] (t) -- node[above,sloped] {} (g);
    \draw[line width=1.5pt,brown, decoration={markings,mark=at position 1 with
    {\arrow[scale=1,>=stealth]{>}}}, postaction={decorate}] (g) -- node[above,sloped] {} (w);
    \end{tikzpicture}

    \vspace{0.2cm}
    \footnotesize{a)} \hspace{3.2cm} \footnotesize{b)} \hspace{3.2cm} \footnotesize{c)}
    \vspace{0.6cm}

    \begin{tikzpicture}[shorten >=2pt, scale=0.65]
    \path (-0.5,-0.5) node[circle,,draw] (x) {   $1$}
    (0,1.6) node[circle,,draw] (k) {   $2$}
    (1.2,2.9) node[circle,draw](y) {   $3$}
    (1.2,-1.3) node[circle,draw](z) {   $8$}
    (3,2.9) node[circle,draw](w) {   $4$}
    (3,-1.3) node[circle,draw](t) {   $7$}
    (4.2,0) node[circle,draw](g) {   $6$}
    (4.7,2.1) node[circle,draw](m) {   $5$};

    \draw[line width=1.5pt,red, decoration={markings,mark=at position 1 with
    {\arrow[scale=1,>=stealth]{>}}}, postaction={decorate}] (k) -- node[above,sloped] {} (z);
    \draw[line width=1.5pt,red, decoration={markings,mark=at position 1 with
    {\arrow[scale=1,>=stealth]{>}}}, postaction={decorate}] (y) -- node[above,sloped] {} (k);
    \draw[line width=1.5pt,red, decoration={markings,mark=at position 1 with
    {\arrow[scale=1,>=stealth]{>}}}, postaction={decorate}] (z) -- node[above,sloped] {} (t);
    \draw[line width=1.5pt,brown, decoration={markings,mark=at position 1 with
    {\arrow[scale=1,>=stealth]{>}}}, postaction={decorate}] (w) -- node[above,sloped] {} (y);
    \draw[line width=1.5pt,brown, decoration={markings,mark=at position 1 with
    {\arrow[scale=1,>=stealth]{>}}}, postaction={decorate}] (y)  .. controls (2.5,1.8) .. (t);

    \draw[line width=1.5pt,red, decoration={markings,mark=at position 1 with
    {\arrow[scale=1,>=stealth]{>}}}, postaction={decorate}] (t) .. controls (1,0.5) .. (y);
    \draw[line width=1.5pt,brown, decoration={markings,mark=at position 1 with
    {\arrow[scale=1,>=stealth]{>}}}, postaction={decorate}] (t) -- node[above,sloped] {} (g);
    \draw[line width=1.5pt,brown, decoration={markings,mark=at position 1 with
    {\arrow[scale=1,>=stealth]{>}}}, postaction={decorate}] (g) -- node[above,sloped] {} (w);
    
    \path (7,-0.5) node[circle,,draw] (x) {   $1$}
    (7.5,1.6) node[circle,,draw] (k) {   $2$}
    (8.7,2.9) node[circle,draw](y) {   $3$}
    (8.7,-1.3) node[circle,draw](z) {   $8$}
    (10.5,2.9) node[circle,draw](w) {   $4$}
    (10.5,-1.3) node[circle,draw](t) {   $7$}
    (11.7,0) node[circle,draw](g) {   $6$}
    (12.2,2.1) node[circle,draw](m) {   $5$};
    
    \draw[line width=1.5pt,black, decoration={markings,mark=at position 1 with
    {\arrow[scale=1,>=stealth]{>}}}, postaction={decorate}] (x) -- node[above,sloped] {} (k);
    \draw[line width=1.5pt,black, decoration={markings,mark=at position 1 with
    {\arrow[scale=1,>=stealth]{>}}}, postaction={decorate}] (k) -- node[above,sloped] {} (z);
    \draw[line width=1.5pt,black, decoration={markings,mark=at position 1 with
    {\arrow[scale=1,>=stealth]{>}}}, postaction={decorate}] (z) -- node[above,sloped] {} (x);
    \draw[line width=1.5pt,black, decoration={markings,mark=at position 1 with
    {\arrow[scale=1,>=stealth]{>}}}, postaction={decorate}] (g) -- node[above,sloped] {} (w);
    \draw[line width=1.5pt,black, decoration={markings,mark=at position 1 with
    {\arrow[scale=1,>=stealth]{>}}}, postaction={decorate}] (w) -- node[above,sloped] {} (m);
    \draw[line width=1.5pt,black, decoration={markings,mark=at position 1 with
    {\arrow[scale=1,>=stealth]{>}}}, postaction={decorate}] (m) -- node[above,sloped] {} (g);

    \end{tikzpicture} 
    
    \vspace{0.2cm}
    \footnotesize{d)} \hspace{3.7cm} \footnotesize{e)}
    \vspace*{-0.3cm}
    \end{center}
    
    \caption{An example with opposite non-tour arcs. a) Initial tour $T$ b) Tour-alternating cycle $A$ c) Tour-alternating cycle $A'$ d) Tour-alternating circulation $C$ composed of $A$ and $A'$ e) Flow formed by $C$ with a trim of opposite non-tour arcs \label{fig:circu1}}
\end{figure}

Most of the previous arguments transfer to this setting, but there is an additional technical challenge. Figure~\ref{fig:circu1} illustrates an example where a tour–alternating circulation makes use of opposite non–tour arcs between a pair of vertices. Figure \ref{fig:circu1}a) shows an initial tour $T$; Figures  \ref{fig:circu1}b) and c) show two tour-alternating cycles $A,A'$ in $R$; and Figure \ref{fig:circu1}d) shows their union, a tour-alternating circulation $C$. The opposite non-tour arcs $(7,3)$ and $(3,7)$ lie in $A$ and $A'$, respectively, and thus in $C$. The effect is that both pairs of removal arcs incident to vertices $3$ and $7$ are used.

There are two viable ways to deal with this situation: (i) Either one allows opposite non–tour arcs $(u,v)$ and $(v,u)$ to both be placed in the post–cancel flow; then they form a subtour by themselves. (ii) Or one treats the pair as an undesirable $2$–cycle -- it lies fully outside the original tour and is of positive cost on both arcs -- and trims it from the support. We choose this latter approach. We delete both insertion arcs $(u,v)$ and $(v,u)$, which preserves local balance at $u$ and $v$ (one incoming and one outgoing arc are removed at each) and leaves the remaining support as a collection of subtours together with any vertices that may have become \emph{isolated}. The result is depicted in Figure \ref{fig:circu1}e). In Section~\ref{sec:tailoredclassic}, we tailor the classic patching idea to also reconnect isolated vertices to proper subtours.

\begin{corollary}\label{cor:single-cycle}
Let $R=(G,T)$ be a tour graph and let $C$ be a tour-alternating circulation in $R$. Then the cancellation of $C$ (with trimming of opposite non-tour arcs) yields a feasible $0,1$-circulation that is either (i) a tour, (ii) a collection of subtours, or (iii) a 
collection of subtours together with pairs of isolated vertices for each trim of opposite non-tour arcs.
\end{corollary}

\begin{proof}
Recall that $C$ decomposes into (possibly non-simple) arc-disjoint tour-alternating cycles $A_1,\dots,A_q$. By the same arguments as in the proof of Theorem~\ref{thm:single-cycle} for individual cycles, at every vertex visited by a cycle $A_i$, each visit contributes one insertion arc and one removal arc, maintaining an inflow and outflow of one. Due to arc-disjointness of $A_1\dots,A_q$, the post-cancel flow is a feasible (and spanning) $0,1$–circulation with inflow and outflow of one at all vertices, and thus a tour or a collection of subtours. 

When opposite non–tour arcs $(u,v)$ and $(v,u)$ occur, trimming this directed $2$–cycle deletes one incoming and one outgoing insertion arc at both $u$ and $v$. This preserves flow balance, but leaves $u$ and $v$ as isolated vertices. 
\end{proof}

In all situations, a cancel step along a tour–alternating cycle or circulation leaves a feasible $0,1$–circulation that stays close to a tour: either a tour outright, a (small) collection of subtours, or a collection of subtours with a few isolated vertices ready to be absorbed. Moreover, whenever a cancel step with a single, simple cycle produces a new tour, the cycle must be tour–alternating; and if a cancel step with a non–simple cycle or a circulation produces a tour, its effect decomposes into arc–disjoint tour–alternating cycles. Together, these facts justify restricting the search space to tour–alternating structures: One searches for an improvement in cost while remaining in the proximity of a tour, and any remaining work is confined to a lightweight patching step. In the following section, we devise a strategy to efficiently find an improving tour-alternating cycle or circulation.

\subsection{Computational Strategy}\label{sec:compstrategy}

We detect tour–alternating structures by working on a bipartite \emph{separated graph} that encodes, by construction, the alternation between removal and insertion arcs of the tour graph $R=(G,T)$. We begin with the formal construction.

\begin{definition}[Separated graph]\label{def:separated}
Let $R=(G,T)$ be a tour graph on $V=\{1,\dots,n\}$ with residual arc sets $E^{+}$ and $E^{-}$. The \emph{separated graph} is the directed bipartite graph
\[
R'=(A\cup B,\;E',\;c^{R'}),
\quad
A=\{1,\dots,n\},\ \ B=\{1',\dots,n'\},
\]
with arc set 
\[
E'=\{(h,\,k'):\ (h,k)\in E^{+}\}\ \cup\ \{(\pi(i)',\,i):\ (\pi(i),i)\in E^{-}\}
\]
and costs
\[
c^{R'}_{h,\,k'}=c_{h,k}\quad\text{for }(h,k) \in E^+, 
\qquad
 c^{R'}_{\pi(i)',\,i}=-c_{i,\pi(i)}\quad\text{for }(\pi(i),i)\in E^{-}.
\]

\end{definition}

Informally, we duplicate the vertex set $V$, with original copy $A$ and duplicate copy $B$. All insertion arcs are oriented from $A$ to $B$, and all removal arcs are oriented from $B$ to $A$. Thus $R'$ is a bipartite graph, and every directed cycle in $R'$ alternates between $A$ and $B$ and hence between insertion arcs ($E^{+}$) and removal arcs ($E^{-}$) when interpreted back on $R$. Costs of the arcs are mirrored, as well. Thus, under the inverse mapping back to $R$, every directed cycle $C'$ in $R'$ corresponds to a tour–alternating cycle $C$ in $R$ with the same cost:
\begin{align*}
c^{R'}(C')=\sum_{e'\in C'} c^{R'}(e')
& \;=\; 
\sum_{(h,k')\in C'} c_{h,k'}^{R'} + \sum_{(\pi(i)',i)\in C'} c^{R'}_{\pi(i)',\,i}
\;=\;\\
& \;=\; \sum_{(h,k)\in C\cap E^+} c_{h,k}
\;-\!\!\sum_{(\pi(i),i)\in C\cap E^-} c_{i,\pi(i)}\;=\;\\
& \;=\; \sum_{(h,k)\in C\cap E^+} c^R_{h,k}
\;+\!\!\sum_{(\pi(i),i)\in C\cap E^-} c^R_{\pi(i),i}
\;=\;\sum_{e\in C} c^{R}(e)
\;=\; c^{R}(C).
\end{align*}

Figure~\ref{fig:sepgraph} visualizes the construction. Figure~\ref{fig:sepgraph}a) shows a tour graph $R=(G,T)$, and Figure~\ref{fig:sepgraph}b) the corresponding separated graph $R'$. Let us summarize the observed properties.

\begin{figure}[h]
    \begin{center}
\begin{tikzpicture}[shorten >=2pt, scale=0.75]

    \draw[line width=1pt,>=triangle 45, black, <->] (2.5,2) -- (5,2);
        
    \path 
    (-2.2,4) node[circle,draw](y) {   $1$}
    (0,4) node[circle,draw](z) {   $2$}
    (1.3,2) node[circle,draw](w) {   $3$}
    (0,0) node[circle,draw](t) {   $4$}
    (-2.2,0) node[circle,draw](g) {   $5$};

    \draw[line width=1pt,decoration={markings,mark=at position 1 with
    {\arrow[scale=1.5,>=stealth]{>}}}, red] (y) -- (w);
    \draw[line width=1pt,decoration={markings,mark=at position 1 with
    {\arrow[scale=1.5,>=stealth]{>}}}, red] (y) -- (t);
    \draw[line width=1pt,decoration={markings,mark=at position 1 with
    {\arrow[scale=1.5,>=stealth]{>}}}, red] (z) -- (t);
    \draw[line width=1pt,decoration={markings,mark=at position 1 with
    {\arrow[scale=1.5,>=stealth]{>}}}, red] (z) -- (g);
    \draw[line width=1pt,decoration={markings,mark=at position 1 with
    {\arrow[scale=1.5,>=stealth]{>}}}, red] (w) -- (g);
    
    \draw[line width=1pt,blue, decoration={markings,mark=at position 1 with
    {\arrow[scale=1,>=stealth]{>}}}, postaction={decorate}] (z) -- (y);
    \draw[line width=1pt,blue, decoration={markings,mark=at position 1 with
    {\arrow[scale=1,>=stealth]{>}}}, postaction={decorate}]  (w) -- (z);
    \draw[line width=1pt,blue, decoration={markings,mark=at position 1 with
    {\arrow[scale=1,>=stealth]{>}}}, postaction={decorate}]  (t) -- (w);
    \draw[line width=1pt,blue, decoration={markings,mark=at position 1 with
    {\arrow[scale=1,>=stealth]{>}}}, postaction={decorate}] (g) -- (t);
    \draw[line width=1pt,blue, decoration={markings,mark=at position 1 with
    {\arrow[scale=1,>=stealth]{>}}}, postaction={decorate}] (y) -- (g);

    \path (6,6) node[circle,draw] (x) {   $1$}
    (6,4) node[circle,draw](y) {   $2$}
    (6,2) node[circle,draw](z) {   $3$}
    (6,0) node[circle,draw](w) {   $4$}
    (6,-2) node[circle,draw](t) {   $5$}
    (10,6) node[circle,draw] (u) {   $1'$}
    (10,4) node[circle,draw](v) {   $2'$}
    (10,2) node[circle,draw](h) {   $3'$}
    (10,0) node[circle,draw](i) {   $4'$}
    (10,-2) node[circle,draw](j) {   $5'$};
    
    \draw[line width=1pt,red, decoration={markings,mark=at position 1 with
    {\arrow[scale=1.5,>=stealth]{>}}}, postaction={decorate}] (x) -- (h);
    \draw[line width=1pt,red, decoration={markings,mark=at position 1 with
    {\arrow[scale=1.5,>=stealth]{>}}}, postaction={decorate}] (x) -- (i);
    \draw[line width=1pt,red, decoration={markings,mark=at position 1 with
    {\arrow[scale=1.5,>=stealth]{>}}}, postaction={decorate}] (x) -- (j);
    \draw[line width=1pt,red, decoration={markings,mark=at position 1 with
    {\arrow[scale=1.5,>=stealth]{>}}}, postaction={decorate}] (y) -- (i);
    \draw[line width=1pt,red, decoration={markings,mark=at position 1 with
    {\arrow[scale=1.5,>=stealth]{>}}}, postaction={decorate}] (y) -- (j);
    \draw[line width=1pt,red, decoration={markings,mark=at position 1 with
    {\arrow[scale=1.5,>=stealth]{>}}}, postaction={decorate}] (z) -- (u);
    \draw[line width=1pt,red, decoration={markings,mark=at position 1 with
    {\arrow[scale=1.5,>=stealth]{>}}}, postaction={decorate}] (z) -- (j);
    \draw[line width=1pt,red, decoration={markings,mark=at position 1 with
    {\arrow[scale=1.5,>=stealth]{>}}}, postaction={decorate}] (w) -- (u);
    \draw[line width=1pt,red, decoration={markings,mark=at position 1 with
    {\arrow[scale=1.5,>=stealth]{>}}}, postaction={decorate}] (w) -- (v);
    \draw[line width=1pt,red, decoration={markings,mark=at position 1 with
    {\arrow[scale=1.5,>=stealth]{>}}}, postaction={decorate}] (t) -- (u);
    \draw[line width=1pt,red, decoration={markings,mark=at position 1 with
    {\arrow[scale=1.5,>=stealth]{>}}}, postaction={decorate}] (t) -- (h);
    \draw[line width=1pt,red, decoration={markings,mark=at position 1 with
    {\arrow[scale=1.5,>=stealth]{>}}}, postaction={decorate}] (t) -- (v);

    \draw[line width=1pt,blue, decoration={markings,mark=at position 1 with
    {\arrow[scale=1,>=stealth]{>}}}, postaction={decorate}] (v) -- (x);
    \draw[line width=1pt,blue, decoration={markings,mark=at position 1 with
    {\arrow[scale=1,>=stealth]{>}}}, postaction={decorate}] (h) -- (y);
    \draw[line width=1pt,blue, decoration={markings,mark=at position 1 with
    {\arrow[scale=1,>=stealth]{>}}}, postaction={decorate}] (i) -- (z);
    \draw[line width=1pt,blue, decoration={markings,mark=at position 1 with
    {\arrow[scale=1,>=stealth]{>}}}, postaction={decorate}] (j) -- (w);
    \draw[line width=1pt,blue, decoration={markings,mark=at position 1 with
    {\arrow[scale=1,>=stealth]{>}}}, postaction={decorate}] (u) -- (t);
    
    \end{tikzpicture}  
    
    \vspace{0.3cm}
    
    \footnotesize{a)} \hspace{5.8cm} \footnotesize{b)}
    
    \vspace*{-0.3cm}
    \end{center}
    
    \caption{Separated graph construction. a) Tour graph $R=(G,T)$. b) Separated graph $R'$.}\label{fig:sepgraph}
\end{figure}
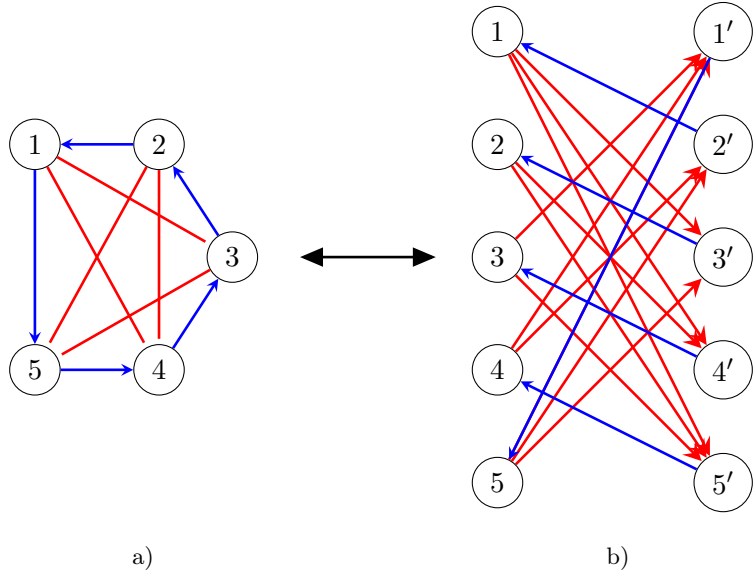

\begin{theorem}[Separated Graph Correspondence]\label{thm:sepgraph-correspondence}
Let $R'=(A\cup B,E',c^{R'})$ be the separated graph of a tour graph $R=(G,T)$. Then
\begin{enumerate}
\item \textbf{Alternation.} Every directed cycle $C'$ in $ R'$ alternates between $A$ and $B$.
\item \textbf{Cycle mapping.} Replacing each arc $(h,k')$ of $C'$ by $(h,k)$ and each arc $(\pi(i)',i)$ by $(\pi(i),i)$ yields a tour–alternating cycle $C$ in $R$. This defines a bijection between directed cycles $C'$ in $R'$ and tour–alternating cycles $C$ in $R$.
\item \textbf{Cost preservation.} For any cycle $C'$ in $ R'$ with image $C$ in $R$,
$
c^{R'}(C')\;=\; c^{R}(C).
$
\end{enumerate}
The same statements extend cyclewise to circulations, with total costs preserved.
\end{theorem}

Our goal is to apply (tailored versions of) classic algorithms such as Floyd-Warshall for negative cycle (or circulation) detection on the separated graph. These routines take as input a distance matrix $C=(c_{i,j})$ and predecessor matrix $P$. We conclude this section with the simple matrix-based construction of the separated graph, summarized in Algorithm~\ref{alg:transf}, and record the computational cost.

\begin{lemma}\label{lem:sep-construct}
Let $C=(c_{i,j})$ be an $n\times n$ distance matrix for a graph $G$ and let $T$ be a tour.
The construction of the $2n\times 2n$ distance matrix $D$ and predecessor matrix $P$ for the separated graph $R'$ runs in $\Theta(n^2)$ time in the arithmetic model of computation.
\end{lemma}

\begin{proof}
In the arithmetic model of computation, it suffices to count elementary operations. Algorithm~\ref{alg:transf} initializes the $2n\times 2n$ matrices $D$ and $P$ in line~3, which takes $\Theta(n^2)$ time.
The nested loops over $i$ and $j$ (lines~4--5) lead to exactly $n^2$ iterations.
In each iteration, the code executes a constant amount of work: it skips over cases $i = j$ (lines~6--7) and, for $i \neq j$, checks whether $(i,j)\in T$ and then writes one entry in each of $D$ and $P$ depending on whether  $(i,j)\in T$ (lines~8--9) or $(i,j)\notin T$ (lines~10--11). Therefore the total running time is $\Theta(n^2)$.
\end{proof}

\begin{algorithm}[t]\small
\caption{Matrix Construction for Separated Graph}\label{alg:transf}
\begin{algorithmic}[1]
\State \textbf{Input:} Distance matrix $C$ for graph $G$, tour $T$
\State \textbf{Output:} Distance matrix $D$, predecessor matrix $P$ for separated graph $R'$
\State Initialize $2n\times 2n$ matrices $D$ and $P$ with all entries $+\infty$
\For{$i=1$ \textbf{to} $n$}
  \For{$j=1$ \textbf{to} $n$}
    \If{$i=j$} \State \textbf{continue}
    \ElsIf{$(i,j)\in T$}
      \State $D[n+j,\,i]\gets -c_{i,j}$; \quad $P[n+j,\,i]\gets n+j$
    \Else
      \State $D[i,\,n+j]\gets c_{i,j}$; \quad $P[i,\,n+j]\gets i$
    \EndIf
  \EndFor
\EndFor
\State \textbf{return} $D,P$
\end{algorithmic}
\end{algorithm}

In view of \textsc{Cycap} as a whole, the quadratic running time of the graph transformation, i.e., of Algorithm~\ref{alg:transf}, is dominated by the running times of the negative cycle and circulation routines applied to the separated graph, which are cubic and super-quadratic, respectively. 

\section{Implementation}\label{sec:implementation}
This section provides a practical perspective on the proposed heuristic. In Section \ref{sec:variants}, we begin by describing the three variants of \textsc{Cycap} and highlight differences in cycle detection and circulation strategies. In Section \ref{sec:tailoredclassic}, we then provide information on tailored implementations of classic algorithms that are used as subroutines.

\subsection{Variants of the \textsc{Cycap} Heuristic}\label{sec:variants}

The \textsc{Cycap} heuristic was implemented in three variants -- \textsc{Cycap-F}, \textsc{Cycap-M}, and \textsc{Cycap-C} -- each differing in how cycles and circulations are identified. The source code for all three \textsc{Cycap} variants, along with scripts to reproduce our experiments in Section \ref{sec:experiments}, is available at \url{https://github.com/zsorenso/CCTP}. For some background on classic negative-cycle finding and minimum-cost circulation strategies, see for example \cite{o-93}.\\\\
\noindent{\bf Cycap-F (Floyd–Warshall-based approach)}
\textsc{Cycap-F} applies the Floyd–Warshall algorithm to the separated graph to detect negative cycles. Unlike a standard early-stop approach, where the algorithm terminates upon detection of a negative cycle due to the ability of returning to the same vertex at negative cost, \textsc{Cycap-F} runs the algorithm to completion, i.e., it continues to perform iterations until all vertices have been considered. A predecessor matrix is stored and updated in each iteration. After the algorithm finishes, an adjusted \emph{Predecessor Readout} routine is applied to every stored matrix row with a negative diagonal entry. This routine systematically backtracks through the stored predecessor matrix to identify negative cycles for all vertices identified to lie on some negative cycle, trimming extraneous vertices and collecting all cycles for evaluation. Each detected cycle is translated into a tour-alternating cycle in the tour graph for cancellation (and patch). The heuristic evaluates all resulting candidate tours, and returns the cheapest one. This design allows \textsc{Cycap-F} to leverage multiple improving cycles rather than a single update. For details on the predecessor readout routine, including pseudocode, see Section~\ref{sec:tailoredclassic}.

There is a tradeoff between computational speed and the search for multiple negative cycles: Floyd–Warshall has a $O(n^3)$ worst-case running time for a graph with $n$ vertices, but in its standard implementation with termination upon finding a first negative cycle often stops early in practice. Our approach is in $\Theta(n^3)$ due to performing a complete run, making this variant computationally intensive for large instances.\\\\ 
\noindent{\bf Cycap-M (Minimum-mean cycle approach)}
\textsc{Cycap-M} adopts a different strategy by searching for a cycle with the minimum-mean cost rather than any negative cycle. A minimum-mean cycle is a cycle whose total cost divided by its number of arcs is lowest among all cycles in the graph. This criterion prioritizes (short) cycles that offer the greatest improvement per arc used. In Section \ref{sec:conclusion}, we briefly discuss some of the benefits of cycles with few arcs in \textsc{Cycap}.

Minimum-mean cycle canceling, in Karp's implementation, runs in $O(m \cdot n)$, where $m$ is the number of arcs. This is asymptotically faster than Floyd–Warshall for sparse graphs, but the algorithm comes with more computational effort in each iteration. Additionally, in our proof-of-concept implementation we use a distance matrix, a dense data structure, to store information on the separated graph (see Algorithm \ref{alg:transf} in Section \ref{sec:compstrategy}), so we are not making direct use of some of the possible advantage of sparsity.
\\\\
\noindent{\bf Cycap-C (Circulation-based approach)} \textsc{Cycap-C} differs from the other variants by solving a minimum-cost circulation problem rather than searching for individual cycles. This approach computes a global flow adjustment that may simultaneously cancel multiple negative cycles, rather than performing a single-cycle improvement. The circulation problem can be solved with specialized network flow algorithms or linear programming, with practical algorithms having running times on the order of $O(m \cdot \text{polylog}\; n)$. In our implementation, we solve the circulation problem on the separated graph via linear programming. As we will see, \textsc{Cycap-C} consistently and significantly outperforms the other variants in running time on large instances. 

We note that identifying the globally-optimal negative cycle -- i.e., the cycle with the smallest total cost -- is NP-hard in general. As a result, \textsc{Cycap‑F} (which relies on Floyd–Warshall to find some negative cycle) and \textsc{Cycap‑M} (which targets a minimum-mean negative cycle) do not guarantee finding the most cost-reducing cycle possible. In contrast, \textsc{Cycap‑C} formulates a minimum-cost circulation problem for the separated graph, which can be solved exactly. This means \textsc{Cycap‑C} effectively identifies an \emph{optimal} circulation; each individual cycle in this circulation is of negative cost, and together they are a globally optimal collection of arc-disjoint cycles. This is a key reason for the particularly good performance of \textsc{Cycap‑C} we will observe in our experiments. 

\subsection{Tailored Versions of Classic Algorithms}\label{sec:tailoredclassic}

This section collects the implementation-level adaptations of classic tools that we use as subroutines within \textsc{Cycap}. First, we describe an adjusted \emph{$k^{*}$-opt Algorithm} for directed instances, which explicitly accounts for the asymmetry created by segment reversals and therefore yields stronger local baselines. Next, we explain an adjusted \emph{Predecessor Readout}, a variant of Floyd-Warshall predecessor readout designed to harvest several negative cycles rather than only detect the existence of one (and possibly returning that first cycle found), providing a collection of candidate cycles for cancellation in \textsc{Cycap-F}. Finally, we describe an adjusted \emph{Patching Routine}, an iterative scheme that merges the subtours (and occasional isolated vertices) produced by a cancellation step into a single tour by greedily selecting cheapest patches.\\

\begin{algorithm}[t]\small
\caption{$k^*$-opt Algorithm (Directed TSP)}\label{alg:kstar-opt}
\begin{algorithmic}[1]
\Require Distance matrix $C$ for graph $G$, tour $T$, parameter $k \in \mathbb{N}$
\Ensure Improved tour $T$
\State \textbf{repeat}
    \State $\textit{improved} \gets \textbf{false}$
    \ForAll{choices of $k$ arcs $E_k$ in $T$}
        \State Remove $E_k$ from $T$ to create $k$ path segments
        \State $\mathcal{S}_k \gets$ all $k$-opt reconnections of the $k$ segments into a tour
        \State $\mathcal{S}_k^{\text{rev}} \gets \{\,\operatorname{rev}(U) : U \in \mathcal{S}_k \,\}$
        \State $S \gets \{T\} \cup \mathcal{S}_k \cup \mathcal{S}_k^{\text{rev}} \cup \{\operatorname{rev}(T)\}$
        \State $U^{\star} \gets \arg\min \{\, c(U) : U \in S \,\}$ 
        \If{$c(U^{\star}) < c(T)$}
            \State $T \gets U^{\star}$ 
            \State $\textit{improved} \gets \textbf{true}$
            \State \textbf{break} 
        \EndIf
    \EndFor
\State \textbf{until} $\textit{improved} = \textbf{false}$ 
\State \Return $T$
\end{algorithmic}
\end{algorithm}

\noindent{\bf k$^*$-opt Algorithm (Directed TSP).}
For directed instances, as an alternative to classic $k$-opt, we use an adjusted version, denoted $k^*$-opt, to account for asymmetry in arc costs when segments are reversed. See Algorithm \ref{alg:kstar-opt} for a description in pseudocode. 

Let $T$ be the current tour and $c(T)$ its cost. Further, let $E_k$ denote a choice of $k$ tour arcs (line~3). We denote by $\mathcal{S}_k(T)$ the set of directed candidate tours formed by removing the arcs in $E_k$ and reconnecting the path segments, as in an iteration of standard $k$-opt (lines~4--5). The $k^*$-opt rule augments this set by also evaluating the reversed orientation $\operatorname{rev}(U)$ of each candidate $U$, as well as the reversed orientation $\operatorname{rev}(T)$ of $T$ itself (lines~6--7):
\[
\mathcal{S}_k^{*}(T) \;=\; \mathcal{S}_k(T)\,\cup\,\{\operatorname{rev}(U)\;:\;U\in \mathcal{S}_k(T)\}\,\cup\,\{\operatorname{rev}(T)\}.
\]
The update then selects
\[
T^{\text{new}} \;=\; \arg\min\Big\{\, c(U)\;:\; U \in \mathcal{S}_k^{*}(T)\,\Big\},
\]
and accepts $T^{\text{new}}$ if and only if $c(T^{\text{new}})<c(T)$ (lines~8--13). Concretely, for a choice $E_2$ of two arcs or a choice $E_3$ of three arcs, $2^*$-opt compares $T$ to three tours $\{T',\,\operatorname{rev}(T'), \operatorname{rev}(T)\}$ where $T'$ is the unique tour in $S_2(T)$, while $3^*$-opt compares the original tour $T$ to the $7$ standard reconnections $S_k(T)$ together with their $7$ reversals and $\operatorname{rev}(T)$, for $15$ candidates in total. 
\\\\
\noindent{\bf Predecessor Readout Routine.} 
Classic implementations of the Floyd-Warshall algorithm (see for example \cite{amo-93}) only detect existence of a (single) negative cycle and only provide a canonical reconstruction of that first cycle found. Most implementations early–stop once a negative diagonal entry $D[i,i]$ appears. 
The predecessor matrix $P$ upon such an early-stop allows a readout of the cycle by starting at $P[i,i]$ and backtracking through predecessors until vertex $i$ is reached again. 

An adjusted predecessor readout is used by \textsc{Cycap-F} to extract negative cycles from the predecessor matrix produced by a complete run of Floyd–Warshall, i.e., a run that does not stop when diagonal entries become negative. The pseudocode is shown in Algorithm \ref{alg:adjreadout}. Because of performing a complete run, we have the opportunity to find multiple negative cycles. However, accessing them becomes more technical: If a shortest walk from vertex $i$ to vertex $j$ contains a negative cycle, the standard backtracking process might encounter an intermediate vertex twice, and vertex $i$ would never be reached. This also holds for $i=j$.

Our adjusted readout iterates over all vertices with negative diagonal entries in the final distance matrix (lines~2-7), which are guaranteed to lie on a negative cycle, and also iterates over all vertices $j$ (line~8), to potentially find additional negative cycles. For each combination of $i$ and $j$ and the associated walk, it backtracks through predecessors to identify negative cycles (lines~10, 16--17), which is successful when a vertex is encountered twice (line~11). When that happens, the routine trims non-cycle vertices and stores the cycle (lines~12--14). Multiple cycles can be returned.\\

\begin{algorithm}[t]\small
\caption{Predecessor Readout}\label{alg:adjreadout}
\begin{algorithmic}[1]
\Require Distance matrix $D$, predecessor matrix $P$ formed by Floyd-Warshall algorithm 
\State $L^* \gets$ empty list of cycles, $L \gets \emptyset$
\For{$i = 1$ to $n$}
    \If{$D[i,i] < 0$}
        \State Store $i$
    \EndIf
\EndFor
\For{each stored $i$}
    \For{$j = 1$ to $n$}
        \State $k \gets i$
        \While{$P[i,j] \neq k$}
            \If{$P[i,j]$ already in $L$}
                \State Remove all values in $L$ that precede $P[i,j]$
                \State Store $L$ in $L^*$
                \State Clear $L$ and exit while loop
            \EndIf
            \State Insert $P[i,j]$ into $L$
            \State $i \gets P[i,j]$
        \EndWhile
    \EndFor
\EndFor
\State \Return $L^*$
\end{algorithmic}
\end{algorithm}

\noindent{\bf Patching Routine.} After canceling a tour–alternating cycle or circulation, the resulting flow may decompose into several subtours (and possibly isolated vertices). An adjusted patching routine reconnects these components into a single tour by iteratively applying the cheapest available \emph{patch}, the operation in Karp's classic algorithm \cite{Patch-78}.

A patch of two distinct subtours $S_1, S_2$ removes two arcs -— one from each of $S_1, S_2$ -— and adds two cross–arcs that merge them. 
The cost of a candidate patch that deletes $(i,j)$ from $S_1$ and $(h,k)$ from $S_2$ and adds $(i,k)$ and $(h,j)$ is
\[
\Delta(C)_{i,j;h,k} \;=\; c_{i,k} + c_{h,j} - c_{i,j} - c_{h,k}.
\]
In the symmetric setting the same formula applies with undirected edge costs. Our routine selects two arbitrary subtours and applies the patch of minimum $\Delta$, repeating until a single tour remains. See Figure \ref{fig:apatch} for an example with three subtours; in the naming of subtours, we use subscripts to track which ones were patched.

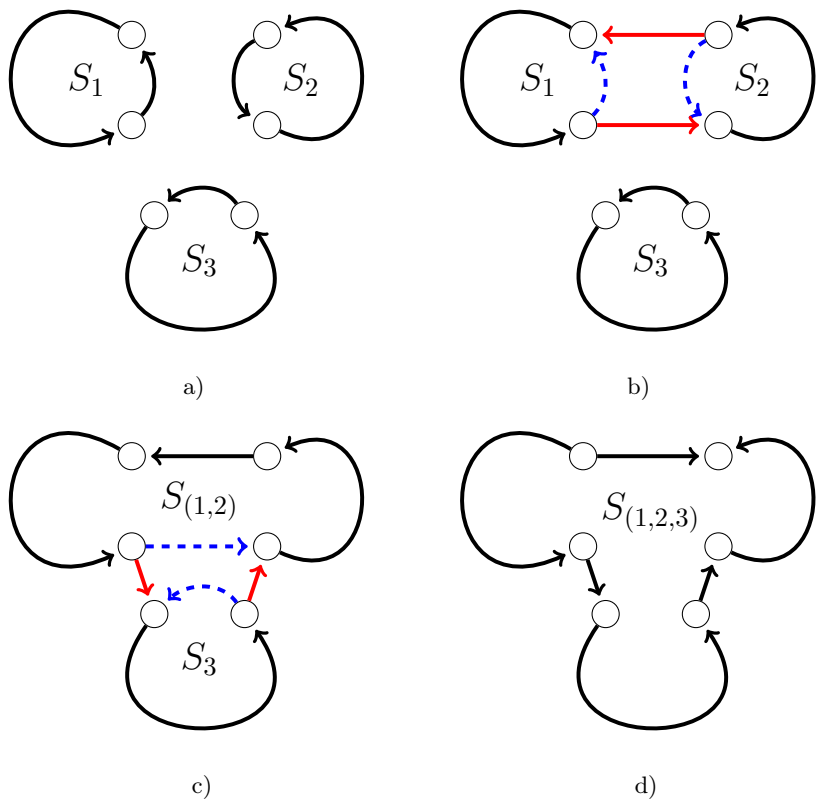
\begin{figure}[h]
\begin{center}
    \begin{tikzpicture}[shorten >=2pt, scale=0.3]
    \path (0,2) node[circle,draw] (a) {  }
    (0,-2) node[circle,draw](b) {   } 
    (6,2) node[circle,draw] (c) {  }
    (6,-2) node[circle,draw](d) {   } 
    (1,-6) node[circle,draw] (e) {  }
    (5,-6) node[circle,draw](f) {   } 
    (-2,0) node (g) {\Large{$S_1$}} 
    (7.5,0) node (h) {\Large{$S_2$}}
    (3,-8) node (i) {\Large{$S_3$}};
    
    \draw[line width=1.5pt,black, font=\bf, ->]  (a) to [out=150, in=205, looseness=4] (b);
    \draw[line width=1.5pt,black, font=\bf, ->]  (b) to [out=45, in=305, looseness=1] (a);
    \draw[line width=1.5pt,black, font=\bf, ->]  (c) to [out=205, in=150, looseness=1] (d);
    \draw[line width=1.5pt,black, font=\bf, ->]  (d) to [out=335, in=25, looseness=3] (c);
    \draw[line width=1.5pt,black, font=\bf, ->]  (e) to [out=235, in=305, looseness=4] (f);
    \draw[line width=1.5pt,black, font=\bf, ->]  (f) to [out=125, in=45, looseness=1] (e);

    \path (20,2) node[circle,draw] (j) {  }
    (20,-2) node[circle,draw](k) {   } 
    (26,2) node[circle,draw] (l) {  }
    (26,-2) node[circle,draw](m) {   } 
    (21,-6) node[circle,draw] (n) {  }
    (25,-6) node[circle,draw](o) {   } 
    (18,0) node (p) {\Large{$S_1$}} 
    (27.5,0) node (q) {\Large{$S_2$}}
    (23,-8) node (r) {\Large{$S_3$}};
    
    \draw[line width=1.5pt,black, font=\bf, ->]  (j) to [out=150, in=205, looseness=4] (k);
    \draw[line width=1.5pt,blue, font=\bf, ->,dashed]  (k) to [out=45, in=305, looseness=1] (j);
    \draw[line width=1.5pt,blue, font=\bf, ->, dashed]  (l) to [out=205, in=150, looseness=1] (m);
    \draw[line width=1.5pt,black, font=\bf, ->]  (m) to [out=335, in=25, looseness=3] (l);
    \draw[line width=1.5pt,black, font=\bf, ->]  (n) to [out=235, in=305, looseness=4] (o);
    \draw[line width=1.5pt,black, font=\bf, ->]  (o) to [out=125, in=45, looseness=1] (n);
    \draw[line width=1.5pt,red, font=\bf, ->]  (k) -- (m);
    \draw[line width=1.5pt,red, font=\bf, ->]  (l) -- (j);
    
    \end{tikzpicture}

        \hspace*{0.5cm}\footnotesize{a)} \hspace{5.4cm} \footnotesize{b)}
    \vspace*{-0.3cm}
    \vspace*{-0.3cm}
    \vspace*{-0.3cm}
   
\end{center}

\begin{center}
    \begin{tikzpicture}[shorten >=2pt, scale=0.3]
    \path (0,2) node[circle,draw] (a) {  }
    (0,-2) node[circle,draw](b) {   } 
    (6,2) node[circle,draw] (c) {  }
    (6,-2) node[circle,draw](d) {   } 
    (1,-5) node[circle,draw] (e) {  }
    (5,-5) node[circle,draw](f) {   } 
    (3,0) node (g) {\Large{$S_{(1,2)}$}} 
    (3,-7) node (i) {\Large{$S_3$}};
    
    \draw[line width=1.5pt,black, font=\bf, ->]  (a) to [out=150, in=205, looseness=4] (b);
    \draw[line width=1.5pt,black, font=\bf, ->]  (d) to [out=335, in=25, looseness=3] (c);
    \draw[line width=1.5pt,black, font=\bf, ->]  (e) to [out=235, in=305, looseness=4] (f);
    \draw[line width=1.5pt,blue, font=\bf, ->, dashed]  (f) to [out=125, in=45, looseness=1] (e);
    \draw[line width=1.5pt,black, font=\bf, ->]  (c) -- (a);
    \draw[line width=1.5pt,red, font=\bf, ->]  (b) -- (e);
    \draw[line width=1.5pt,red, font=\bf, ->]  (f) -- (d);
    \draw[line width=1.5pt,blue, font=\bf, ->, dashed]  (b) -- (d);

    \path (20,2) node[circle,draw] (j) {  }
    (20,-2) node[circle,draw](k) {   } 
    (26,2) node[circle,draw] (l) {  }
    (26,-2) node[circle,draw](m) {   } 
    (21,-5) node[circle,draw] (n) {  }
    (25,-5) node[circle,draw](o) {   } 
    (23,-0.5) node (p) {\Large{$S_{(1,2,3)}$}} 
;
    
    \draw[line width=1.5pt,black, font=\bf, ->]  (j) to [out=150, in=205, looseness=4] (k);
    \draw[line width=1.5pt,black, font=\bf, ->]  (m) to [out=335, in=25, looseness=3] (l);
    \draw[line width=1.5pt,black, font=\bf, ->]  (n) to [out=235, in=305, looseness=4] (o);
    \draw[line width=1.5pt,black, font=\bf, ->]  (j) -- (l);
    \draw[line width=1.5pt,black, font=\bf, ->]  (k) -- (n);
    \draw[line width=1.5pt,black, font=\bf, ->]  (o) -- (m);
    
    \end{tikzpicture}

    \hspace*{0.5cm}\footnotesize{c)} \hspace{5.4cm} \footnotesize{d)}

    \vspace*{-0.3cm}
    
\end{center}

\caption{A patching example. a) Initial set of subtours b) First iteration (patch of $S_1$ to $S_2$) c) Second iteration (patch of $S_{(1,2)}$ to $S_3$) d) Resulting tour $S_{(1,2,3)}$}\label{fig:apatch}
\end{figure}

Unlike in the classic setting, isolated vertices arise naturally when tour-alternating circulations make use of opposite non-tour arcs (recall the discussion in Section \ref{sec:searchspace}), and so we need a special rule. To integrate an isolated vertex $v$ into a subtour $S$, we use a \emph{three–arc} patch: Delete one arc $(h,k)$ from $S$ and add the two arcs $(h,v)$ and $(v,k)$, so that $v$ gains one incoming and one outgoing arc while flow balance is preserved. Figure \ref{fig:isopatch} shows an example. We have 
\[
\Delta(C)_{h,k;v} \;=\; c_{h,v} + c_{v,k} - c_{h,k}.
\]

\begin{figure}[h]
\begin{center}
    \begin{tikzpicture}[shorten >=2pt, scale=0.3]
    \path (0,2) node[circle,draw] (a) {  }
    (0,-2) node[circle,draw](b) {   } 
    (4,0) node[circle,draw] (c) { $v$}
    (-2,0) node (g) {\Large{$S_1$}} ;
    
    \draw[line width=1.5pt,black, font=\bf, ->]  (a) to [out=150, in=205, looseness=4] (b);
    \draw[line width=1.5pt,black, font=\bf, ->]  (b) to [out=45, in=305, looseness=1] (a);
    
    \path (14,2) node[circle,draw] (j) {  }
    (14,-2) node[circle,draw](k) {   } 
    (18,0) node[circle,draw] (l) { $v$ }
    (12,0) node (p) {\Large{$S_1$}};
    
    \draw[line width=1.5pt,black, font=\bf, ->]  (j) to [out=150, in=205, looseness=4] (k);
    \draw[line width=1.5pt,blue, font=\bf, ->,dashed]  (k) to [out=45, in=305, looseness=1] (j);
    \draw[line width=1.5pt,red, font=\bf, ->]  (k) -- (l);
    \draw[line width=1.5pt,red, font=\bf, ->]  (l) -- (j);

    \path (28,2) node[circle,draw] (m) {  }
    (28,-2) node[circle,draw](n) {   } 
    (32,0) node[circle,draw] (o) { }
    (27,0) node (q) {\Large{$S_{(1,v)}$}};
    
    \draw[line width=1.5pt,black, font=\bf, ->]  (m) to [out=150, in=205, looseness=4] (n);
    \draw[line width=1.5pt,red, font=\bf, ->]  (n) -- (o);
    \draw[line width=1.5pt,red, font=\bf, ->]  (o) -- (m);
    
    \end{tikzpicture}

 \vspace*{-0.3cm}

        \hspace*{0.5cm}\footnotesize{a)} \hspace{3.8cm} \footnotesize{b)} \hspace{3.6cm} \footnotesize{c)} \hspace{1 cm}
    
    \vspace*{-0.3cm}
\end{center}

\caption{Patching an isolated vertex. a) Initial subtour $S_1$ and vertex $v$ b) Patching step c) Resulting tour $S_{(1,v)}$}\label{fig:isopatch}
\end{figure}
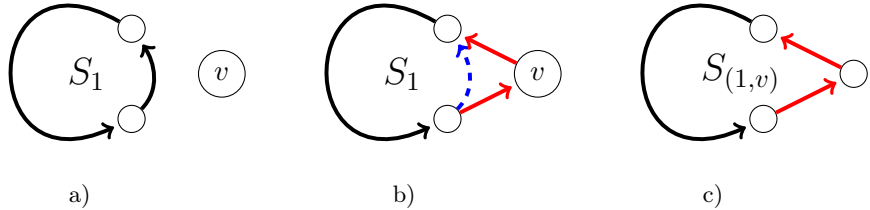
 
The routine proceeds in two phases: (i) If any isolated vertices exist, connect them first using the three–arc rule; otherwise (ii) choose a pair of subtours, evaluate every two–arc patch, and select the cheapest. An efficient implementation makes use of a \emph{successor map} $\pi(\cdot)$ on the current subtours: For each vertex $u$ on a subtour, the arc $(u,\pi(u))$ is present in the flow; this allows a simple rewiring of subtours after a merge through the update of $\pi(i)$ and $\pi(h)$ for the tails $i$ and $h$ of the deleted arcs. Each merge updates the successor map, and the process repeats until no subtours remain. In addition to dealing with isolated vertices, it also differs from Karp's patching routine \cite{Patch-78} in its iterative nature. The classic algorithm forms a locally optimal tour by patching every subtour to the largest initial subtour in a single pass. This approach works the better the longer the initial subtour is. In our experiments, we did not observe a tendency for the cycle cancels to create a large subtour. In each iteration, we perform a best patch among randomly selected subtours. 

Selecting a best patch between two subtours $S_1,S_2$ inspects $O(|S_1|\,|S_2|)$ arc combinations, so each patch runs in $O(n^2)$. 
However, each arc in $G$ is only considered at most once throughout all iterations to take the role of a cross-arc connecting two subtours. Thus, the running time is in $O(m)$ or $O(n^2)$ overall. 

\section{Computational Experiments}\label{sec:experiments}

 In this section, we present computational experiments on benchmark data sets, outline the performance metrics used for evaluation, and discuss observed trends across symmetric and directed TSP instances.

The computational experiments were conducted on benchmark instances from TSPLIB, a widely used repository for evaluating algorithms on the Traveling Salesman Problem \cite{TLIB}. All instances were treated as complete directed graphs, with undesired arcs penalized sufficiently so that they cannot appear in an optimal solution. The Symmetric TSP is readily represented on a directed graph: Each undirected edge is replaced by a pair of opposite arcs of the same cost.

Both symmetric and directed instances vary significantly in size, ranging from small problems with about 20 cities to large-scale cases with over 100 cities. For symmetric TSP experiments, the largest tested instance contained 128 cities, while directed TSP experiments included instances with up to 124 vertices. Each reported measure is based on 100 independent runs per configuration. These problem sizes and repeated trials allow us to assess the performance across different levels of complexity.

We consider an embedding of \textsc{Cycap} with three common variants: $2$-opt, $3$-opt, and a combined $2+3$-opt strategy, which applies both $2$-opt and $3$-opt moves in sequence. The results for directed instances use the $k^*$-opt baselines, as described in Section~\ref{sec:tailoredclassic}, which better reflect the search for a local optimum in that setting.  For directed TSP instances, the running times of complete runs of $k$-opt can become prohibitively large. To ensure fair and practical comparisons, we imposed a time cap on these algorithms: Each $k$-opt run was terminated once it produced a final solution or after it had executed for ten times the median running time of the corresponding \textsc{Cycap} variant. When this limit was hit, the $k$-opt run typically was stuck for a long time at the final, returned tour. This rule was applied consistently across all experiments for directed instances. 

To evaluate performance, we report three primary metrics: \emph{success rate}, \emph{gap closure}, and \emph{running time}. The success rate is the percentage of runs in which the heuristic  improved the initial tour produced by a baseline $k$-opt run. Improvements were considered at two stages: immediately after applying \textsc{Cycap} and after an additional optimization step, denoted by $k+C$ and $k + C+ k$, respectively. This metric provides insight into the ability of the heuristic to escape local optima.

Gap closure quantifies the relative improvement toward the optimal or best known solution. Informally, it is the average fraction of the optimality gap closed by the heuristic. Let \(T_i^o\) denote the cost of the tour after the initial $k$-opt run in experiment \(i\), \(T_i^f\) the cost after applying \textsc{Cycap} (or \textsc{Cycap} plus second $k$-opt), and \(T_{\text{opt}}\) the cost of the optimal or best known tour. The gap closure for run \(i\) is:
\[
\text{GC}_i = \frac{T_i^o - T_i^f}{T_i^o - T_{\text{opt}}}.
\]
The reported value is the average over $n=100$ successful runs, i.e., runs where the $k$-opt tour was improved: 
\[
\text{Gap Closure} = \frac{1}{n} \sum_{i=1}^n \text{GC}_i =  \frac{1}{n} \sum_{i=1}^n \frac{T_i^o - T_i^f}{T_i^o - T_{\text{opt}}}.
\]
For large instances where the globally optimal tour is unknown, \(T_{\text{opt}}\) is approximated by the best tour found in $100$ independent runs of the combined \(2+3+C+2+3\) strategy . High gap closure values indicate that the heuristic achieves substantial reductions in cost for those instances where it is able to find an improvement. 

Finally, running time captures the computational efficiency of the heuristic. We report the time for the heuristic by itself, as well as the total time for embedded runs. This analysis highlights whether the observed improvements come at a practical computational cost. 

When interpreting the upcoming results, it is important to consider the interaction between the \textsc{Cycap} heuristic and the baseline $k$-opt or $k^*$-opt algorithms. Symmetric instances typically allow $k$-opt to quickly reach strong local optima. In contrast, directed instances involve asymmetric arc costs; this makes local optima from $k$-opt weaker, even with the $k^*$-opt variant that we apply, and also increases the running time of $k$-opt and $k^*$-opt itself. In turn, this leads to greater potential for improvement. Reported success rates and gap closures reflect these structural differences, and one should expect a stronger performance of \textsc{Cycap} for directed TSP instances, for which it was primarily designed.

For symmetric problems, we report results for all three \textsc{Cycap} variants: \textsc{Cycap-F}, \textsc{Cycap-M}, and \textsc{Cycap-C}. For directed instances, we only report on \textsc{Cycap-F} and \textsc{Cycap-C}, and not \textsc{Cycap-M}. There are two reasons: First, success rate and gap closure of \textsc{Cycap-M} were similar to \textsc{Cycap-F}, but its minimum-mean cycle detection introduced significant computational overhead for directed problems of all sizes, making it impractical; second, as we will see, \textsc{Cycap-C} significantly outperformed both \textsc{Cycap-F} or \textsc{Cycap-M} in this setting.

\subsection{Success Rate}\label{sec:success-rate}

The success rate measures the proportion of runs in which the heuristic improved the initial tour produced by a $k$-opt run. Table~\ref{tab:cycap-success-symmetric} summarizes success rates for all three \textsc{Cycap} variants on symmetric instances, and Table~\ref{tab:cycap-success-directed} reports results for directed instances.

\begin{table}[h!]
\footnotesize
\centering
\caption{Success Rates on Symmetric Instances}
\label{tab:cycap-success-symmetric}
\renewcommand{\arraystretch}{1.2}
\begin{tabular}{
l@{\hspace{7pt}}|
>{\centering\arraybackslash}p{1.8cm}>{\centering\arraybackslash}p{1.8cm}>{\centering\arraybackslash}p{1.8cm}|
>{\centering\arraybackslash}p{1.8cm}>{\centering\arraybackslash}p{1.8cm}>{\centering\arraybackslash}p{2.0cm}
}
\toprule
\textbf{Data Set} & \textbf{$2+C$} & \textbf{$3+C$} & \textbf{$2+3+C$} & \textbf{$2+C+2$} & \textbf{$3+C+3$} & \textbf{$2+3+C+2+3$} \\
\midrule
\multicolumn{7}{l}{\textbf{Cycap-F}} \\
22cities  & 18\% & 24\% & 17\% & 18\% & 25\% & 17\% \\
30cities  & 36\% & 15\% & 10\% & 76\% & 61\% & 19\% \\
59cities  & 41\% & 46\% & 40\% & 43\% & 71\% & 63\% \\
pr76      & 30\% & 26\% & 19\% & 35\% & 34\% & 96\% \\
128cities & 77\% & 88\% & 68\% & 96\% & 100\%& 94\% \\
\midrule
\multicolumn{7}{l}{\textbf{Cycap-M}} \\
22cities  & 36\% & 87\% & 51\% & 37\% & 88\% & 53\% \\
30cities  & 28\% & 24\% & 11\% & 34\% & 64\% & 40\% \\
59cities  & 12\% & 21\% & 9\%  & 15\% & 24\% & 40\% \\
pr76      & 33\% & 14\% & 16\% & 51\% & 15\% & 18\% \\
128cities & 88\% & 93\% & 87\% & 96\% & 100\%& 100\%\\
\midrule
\multicolumn{7}{l}{\textbf{Cycap-C}} \\
22cities  & 25\% & 15\% & 13\% & 26\% & 71\% & 52\% \\
30cities  & 10\% & 5\%  & 4\%  & 15\% & 56\% & 68\% \\
59cities  & 13\% & 14\% & 10\% & 14\% & 73\% & 75\% \\
pr76      & 4\%  & 5\%  & 8\%  & 9\%  & 72\% & 60\% \\
128cities & 7\%  & 2\%  & 2\%  & 9\%  & 84\% & 85\% \\
\bottomrule
\end{tabular}
\end{table}

For symmetric instances (Table~\ref{tab:cycap-success-symmetric}), success rates vary significantly across variants and problem sizes. \textsc{Cycap-F} and \textsc{Cycap-M} generally outperform \textsc{Cycap-C} in $k+C$ runs, particularly on medium and large instances. The gap becomes very significant for the largest data set 128cities, where success rates for \textsc{Cycap-F} and \textsc{Cycap-M} range from 68\% to 93\% and \textsc{Cycap-C} exhibits a low percentage. \textsc{Cycap-C} shows dramatic improvement in $k+C+k$ runs that use $3$-opt, often surpassing 70\% on larger instances. This suggests that circulation-based adjustments are more effective when followed by an additional $3$-opt step.

\begin{table}[h!]
\footnotesize
\centering
\caption{Success Rates on Directed Instances}
\label{tab:cycap-success-directed}
\renewcommand{\arraystretch}{1.2}
\begin{tabular}{
l@{\hspace{10pt}}|
>{\centering\arraybackslash}p{1.8cm}>{\centering\arraybackslash}p{1.8cm}>{\centering\arraybackslash}p{1.8cm}|
>{\centering\arraybackslash}p{1.8cm}>{\centering\arraybackslash}p{1.8cm}>{\centering\arraybackslash}p{2.0cm}
}
\toprule
\textbf{Data Set} & \textbf{$2+C$} & \textbf{$3+C$} & \textbf{$2+3+C$} & \textbf{$2+C+2$} & \textbf{$3+C+3$} & \textbf{$2+3+C+2+3$} \\
\midrule
\multicolumn{7}{l}{\textbf{Cycap-F}} \\
ftv33    & 58\% & 55\% & 58\% & 62\% & 72\% & 91\% \\
ftv35    & 79\% & 88\% & 67\% & 91\% & 92\% & 100\% \\
ftv38    & 73\% & 67\% & 71\% & 92\% & 100\% & 100\% \\
ftv44    & 82\% & 91\% & 82\% & 100\% & 100\% & 100\% \\
ft53     & 94\% & 71\% & 62\% & 100\% & 100\% & 100\% \\
ftv70    & 89\% & 82\% & 78\% & 92\% & 100\% & 100\% \\
kro124p  & 88\% & 91\% & 89\% & 100\% & 100\% & 100\% \\
\midrule
\multicolumn{7}{l}{\textbf{Cycap-C}} \\
ftv33    & 100\% & 66\% & 64\% & 100\% & 96\% & 91\% \\
ftv35    & 100\% & 74\% & 71\% & 100\% & 91\% & 92\% \\
ftv38    & 100\% & 75\% & 86\% & 100\% & 95\% & 91\% \\
ftv44    & 100\% & 100\% & 94\% & 100\% & 100\% & 96\% \\
ft53     & 100\% & 82\% & 90\% & 100\% & 96\% & 92\% \\
ftv70    & 100\% & 100\% & 100\% & 100\% & 100\% & 100\% \\
kro124p  & 100\% & 100\% & 100\% & 100\% & 100\% & 100\% \\
\bottomrule
\end{tabular}
\end{table}

For directed instances (Table~\ref{tab:cycap-success-directed}), Both \textsc{Cycap-F} and \textsc{Cycap-C} achieve very high success rates, often above 75\% for $k+C$ runs and more than $90\%$ for most $k+C+k$ runs. Notably, for \textsc{Cycap-C}, success rates for the $k+C$ runs are already quite similar to the excellent $k+C+k$ results for both variants. The results confirm that the heuristic is particularly well-suited for directed problems, where local optima from $k^*$-opt are weaker and allow more room for improvement. 

\subsection{Gap Closure}\label{sec:gap-closure}

The gap closure measures the relative improvement toward the optimal or best known solution. We compute the fraction of the initial optimality gap closed by the heuristic, averaged over all runs that improved the $k$-opt tour. Tables~\ref{tab:cycap-gap-symmetric} and \ref{tab:cycap-gap-directed} summarize these results for symmetric and directed instances, respectively.

\begin{table}[h!]
\footnotesize
\centering
\caption{Gap Closure on Symmetric Instances}
\label{tab:cycap-gap-symmetric}
\renewcommand{\arraystretch}{1.2}
\begin{tabular}{
l@{\hspace{10pt}}|
>{\centering\arraybackslash}p{2cm}>{\centering\arraybackslash}p{2cm}>{\centering\arraybackslash}p{2cm}|
>{\centering\arraybackslash}p{2cm}>{\centering\arraybackslash}p{2cm}>{\centering\arraybackslash}p{2.2cm}
}
\toprule
\textbf{Data Set} & \textbf{$2+C$} & \textbf{$3+C$} & \textbf{$2+3+C$} & \textbf{$2+C+2$} & \textbf{$3+C+3$} & \textbf{$2+3+C+2+3$} \\
\midrule
\multicolumn{7}{l}{\textbf{Cycap-F}} \\
22cities  & 28.79\% & 16.46\% & 14.73\% & 55.26\% & 77.41\% & 87.76\% \\
30cities  & 27.43\% & 25.84\% & 25.40\% & 37.24\% & 64.41\% & 36.65\% \\
59cities  & 28.33\% & 20.59\% & 23.78\% & 35.38\% & 44.09\% & 33.55\% \\
pr76      & 14.12\% & 11.75\% & 18.71\% & 19.47\% & 32.49\% & 35.59\% \\
128cities & 3.57\%  & 3.42\%  & 4.38\%  & 34.86\% & 90.28\% & 49.75\% \\
\midrule
\multicolumn{7}{l}{\textbf{Cycap-M}} \\
22cities  & 47.26\% & 62.04\% & 73.82\% & 54.77\% & 59.36\% & 70.37\% \\
30cities  & 41.07\% & 49.50\% & 44.24\% & 63.72\% & 65.79\% & 72.14\% \\
59cities  & 25.34\% & 29.08\% & 45.71\% & 32.25\% & 32.81\% & 61.63\% \\
pr76      & 13.44\% & 14.04\% & 9.64\%  & 41.69\% & 30.46\% & 33.39\% \\
128cities & 24.22\% & 17.39\% & 26.37\% & 32.65\% & 53.78\% & 49.81\% \\
\midrule
\multicolumn{7}{l}{\textbf{Cycap-C}} \\
22cities  & 56.56\% & 35.56\% & 22.02\% & 63.66\% & 47.80\% & 96.40\% \\
30cities  & 46.86\% & 35.52\% & 34.21\% & 72.32\% & 92.46\% & 95.14\% \\
59cities  & 37.87\% & 21.73\% & 29.18\% & 73.65\% & 73.74\% & 65.53\% \\
pr76      & 34.59\% & 31.95\% & 22.74\% & 46.34\% & 57.34\% & 38.31\% \\
128cities & 3.57\%  & 19.77\% & 14.11\% & 34.86\% & 67.98\% & 62.69\% \\
\bottomrule
\end{tabular}
\end{table}

For symmetric instances (Table \ref{tab:cycap-gap-symmetric}), gap closure values indicate that improvements are often substantial already right after completion of \textsc{Cycap}, i.e., for the $k+C$ runs, and then improve further through a second $k$-opt run, i.e., for the $k+C+k$ runs. 
 \textsc{Cycap-M} typically achieves slightly higher gap closure than \textsc{Cycap-F} and \textsc{Cycap-C} in $k+C$ runs. \textsc{Cycap-C} then performs the best in $k+C+k$ runs, where it frequently closes more than 60\% of the gap. This suggests that the use of a tour-alternating circulation, as opposed to a single cycle, in this variant, creates more potential for improvement for a second $k$-opt run.

\begin{table}[h!]
\footnotesize
\centering
\caption{Gap Closure on Directed Instances}
\label{tab:cycap-gap-directed}
\renewcommand{\arraystretch}{1.2}
\begin{tabular}{
l@{\hspace{10pt}}|
>{\centering\arraybackslash}p{2cm}>{\centering\arraybackslash}p{2cm}>{\centering\arraybackslash}p{2cm}|
>{\centering\arraybackslash}p{2cm}>{\centering\arraybackslash}p{2cm}>{\centering\arraybackslash}p{2.2cm}
}
\toprule
\textbf{Data Set} & \textbf{$2+C$} & \textbf{$3+C$} & \textbf{$2+3+C$} & \textbf{$2+C+2$} & \textbf{$3+C+3$} & \textbf{$2+3+C+2+3$} \\
\midrule
\multicolumn{7}{l}{\textbf{Cycap-F}} \\
ftv33    & 17.19\% & 16.22\% & 16.39\% & 29.17\% & 34.19\% & 38.56\% \\
ftv35    & 15.16\% & 12.54\% & 14.25\% & 21.55\% & 43.15\% & 27.09\% \\
ftv38    & 12.63\% & 10.19\% & 10.58\% & 20.41\% & 24.83\% & 23.76\% \\
ftv44    & 10.28\% & 11.75\% & 11.93\% & 15.71\% & 32.81\% & 31.75\% \\
ft53     & 5.21\%  & 5.43\%  & 5.50\%  & 12.22\% & 46.94\% & 30.53\% \\
ftv70    & 7.65\%  & 8.80\%  & 9.20\%  & 12.45\% & 32.44\% & 29.21\% \\
kro124p  & 6.48\%  & 6.57\%  & 8.01\%  & 11.19\% & 28.29\% & 17.43\% \\
\midrule
\multicolumn{7}{l}{\textbf{Cycap-C}} \\
ftv33    & 70.55\% & 24.95\% & 30.92\% & 74.40\% & 45.23\% & 63.43\% \\
ftv35    & 67.13\% & 52.27\% & 32.67\% & 70.50\% & 71.41\% & 61.95\% \\
ftv38    & 83.01\% & 52.96\% & 47.28\% & 84.46\% & 76.93\% & 67.51\% \\
ftv44    & 85.64\% & 38.96\% & 58.42\% & 86.13\% & 65.02\% & 72.06\% \\
ft53     & 86.07\% & 50.95\% & 48.91\% & 86.64\% & 66.41\% & 57.91\% \\
ftv70    & 90.21\% & 83.97\% & 79.94\% & 90.67\% & 87.13\% & 81.98\% \\
kro124p  & 87.95\% & 49.99\% & 55.93\% & 89.81\% & 61.11\% & 60.63\% \\
\bottomrule
\end{tabular}
\end{table}

For directed instances (Table \ref{tab:cycap-gap-directed}), \textsc{Cycap-F} achieves modest gap closure, between about 5\% and 17\% in $k+C$ runs and between 11\% and 47\% in $k+C+k$ runs. In contrast, \textsc{Cycap-C} consistently delivers dramatic improvements, often closing 50\% to 80\% of the gap already for the $k+C$ runs, and even more for the $k+C+k$ runs. These results again highlight the advantage of circulation-based adjustments for directed problems. Recall that \textsc{Cycap-C} also performed very well in view of success rate in the directed setting, so we obtain the favorable combination of both strong success rates and gap closure. 

\subsection{Running Time}\label{sec:running-time}

The running time captures the computational effort required by the \textsc{Cycap} heuristic and its integration with $k$-opt algorithms. We first report the time for \textsc{Cycap} alone, followed by the total time for embedded runs that include the initial and final $k$-opt runs.

\paragraph{\textsc{Cycap}-only times}  
Tables~\ref{tab:cycap-only-symmetric} and \ref{tab:cycap-only-directed} show the average computational times for \textsc{Cycap} alone on symmetric and directed instances. The times are based on a start from a local optimum found via $2$-opt, and very similar for a start from $3$-opt or $2+3$-opt. 

\begin{table}[h!]
\footnotesize
\centering
\caption{\textsc{Cycap}-only Running Times (in seconds) for Symmetric Instances}
\label{tab:cycap-only-symmetric}
\renewcommand{\arraystretch}{1.2}
\begin{tabular}{
l@{\hspace{10pt}}
>{\centering\arraybackslash}p{2cm}
>{\centering\arraybackslash}p{2cm}
>{\centering\arraybackslash}p{2cm}
}
\toprule
\textbf{Data Set} & \textbf{Cycap-F} & \textbf{Cycap-M} & \textbf{Cycap-C} \\
\midrule
22cities  & 0.099 & 0.201 & 0.117 \\
30cities  & 0.604 & 0.863 & 0.314 \\
59cities  & 5.165 & 8.753 & 1.556 \\
pr76      & 28.226 & 9.508 & 12.475 \\
128cities & 347.163 & 67.996 & 14.508 \\
\bottomrule
\end{tabular}
\end{table}

\begin{table}[h!]
\footnotesize
\centering
\caption{\textsc{Cycap}-only Running Times (in seconds) for Directed Instances}
\label{tab:cycap-only-directed}
\renewcommand{\arraystretch}{1.2}
\begin{tabular}{
l@{\hspace{10pt}}
>{\centering\arraybackslash}p{2cm}
>{\centering\arraybackslash}p{2cm}
}
\toprule
\textbf{Data Set} & \textbf{Cycap-F} & \textbf{Cycap-C} \\
\midrule
ftv33    & 0.783  & 0.417 \\
ftv35    & 1.067  & 0.631 \\
ftv38    & 1.341  & 0.646 \\
ftv44    & 2.481  & 0.625 \\
ft53     & 4.233  & 1.401 \\
ftv70    & 19.778 & 1.929 \\
kro124p  & 81.316 & 4.786 \\
\bottomrule
\end{tabular}
\end{table}

For symmetric instances (Table~\ref{tab:cycap-only-symmetric}), \textsc{Cycap-C} is the fastest variant, and the difference in scaling becomes pronounced on large datasets. \textsc{Cycap-M} scales better than \textsc{Cycap-F}, starting with medium-sized symmetric problems. For directed instances (Table \ref{tab:cycap-only-directed}), \textsc{Cycap-C} is much faster than \textsc{Cycap-F} (or \textsc{Cycap-M}), by an order of magnitude for the larger problems.

\paragraph{Embedded runs with $k$-opt}  
Tables~\ref{tab:cycap-time-symmetric} and \ref{tab:cycap-time-directed} report the total time for embedded runs, including the initial and final $k$-opt runs.

\begin{table}[h!]
\footnotesize
\centering
\caption{Running Times (in seconds) for Embedded Runs on Symmetric Instances}
\label{tab:cycap-time-symmetric}
\renewcommand{\arraystretch}{1.2}
\begin{tabular}{
l@{\hspace{10pt}}|
>{\centering\arraybackslash}p{2cm}
>{\centering\arraybackslash}p{2cm}
>{\centering\arraybackslash}p{2cm}|
>{\centering\arraybackslash}p{2cm}
>{\centering\arraybackslash}p{2cm}
>{\centering\arraybackslash}p{2.2cm}
}
\toprule
\textbf{Data Set} & \textbf{$2+C$} & \textbf{$3+C$} & \textbf{$2+3+C$} & \textbf{$2+C+2$} & \textbf{$3+C+3$} & \textbf{$2+3+C+2+3$} \\
\midrule
\multicolumn{7}{l}{\textbf{Cycap-F}} \\
22cities  & 0.184 & 0.153 & 0.198 & 0.219 & 0.185 & 0.271 \\
30cities  & 0.561 & 0.442 & 0.677 & 0.691 & 0.979 & 1.488 \\
59cities  & 9.741 & 10.040 & 11.098 & 11.419 & 52.711 & 61.441 \\
pr76      & 41.680 & 42.827 & 51.229 & 57.838 & 77.402 & 127.907 \\
128cities & 150.181 & 323.692 & 489.901 & 453.466 & 528.935 & 501.799 \\
\midrule
\multicolumn{7}{l}{\textbf{Cycap-M}} \\
22cities  & 0.286 & 0.277 & 0.329 & 0.288 & 0.309 & 0.383 \\
30cities  & 1.099 & 1.080 & 1.353 & 1.690 & 1.310 & 1.650 \\
59cities  & 13.636 & 9.127 & 11.098 & 29.568 & 11.485 & 18.420 \\
pr76      & 39.616 & 29.126 & 52.470 & 54.655 & 34.542 & 56.312 \\
128cities & 187.663 & 324.759 & 602.365 & 336.789 & 727.101 & 676.920 \\
\midrule
\multicolumn{7}{l}{\textbf{Cycap-C}} \\
22cities  & 0.331 & 0.461 & 0.524 & 0.387 & 0.516 & 0.633 \\
30cities  & 0.969 & 0.987 & 1.644 & 1.437 & 1.423 & 2.188 \\
59cities  & 6.389 & 5.435 & 31.440 & 20.509 & 19.660 & 34.698 \\
pr76      & 47.440 & 54.480 & 99.586 & 53.006 & 55.710 & 119.579 \\
128cities & 146.683 & 345.069 & 588.004 & 199.546 & 503.362 & 1013.819 \\
\bottomrule
\end{tabular}
\end{table}

\begin{table}[h!]
\footnotesize
\centering
\caption{Running Times (in seconds) for Embedded Runs on Directed Instances}
\label{tab:cycap-time-directed}
\renewcommand{\arraystretch}{1.2}
\begin{tabular}{
l@{\hspace{10pt}}|
>{\centering\arraybackslash}p{2cm}
>{\centering\arraybackslash}p{2cm}
>{\centering\arraybackslash}p{2cm}|
>{\centering\arraybackslash}p{2cm}
>{\centering\arraybackslash}p{2cm}
>{\centering\arraybackslash}p{2.2cm}
}
\toprule
\textbf{Data Set} & \textbf{$2+C$} & \textbf{$3+C$} & \textbf{$2+3+C$} & \textbf{$2+C+2$} & \textbf{$3+C+3$} & \textbf{$2+3+C+2+3$} \\
\midrule
\multicolumn{7}{l}{\textbf{Cycap-F}} \\
ftv33    & 3.582  & 10.031  & 12.506  & 4.975  & 17.976  & 20.102 \\
ftv35    & 4.256  & 11.768  & 13.902  & 5.196  & 20.836  & 23.757 \\
ftv38    & 3.892  & 17.569  & 21.806  & 8.549  & 30.992  & 36.040 \\
ftv44    & 6.429  & 29.741  & 34.528  & 15.259 & 52.516  & 59.623 \\
ft53     & 16.509 & 55.174  & 65.631  & 18.978 & 97.264  & 110.023 \\
ftv70    & 52.305 & 220.713 & 263.416 & 59.463 & 404.245 & 446.948 \\
kro124p  & 457.467 & 901.624 & 1203.694 & 588.654 & 1842.285 & 2064.626 \\
\midrule
\multicolumn{7}{l}{\textbf{Cycap-C}} \\
ftv33    & 2.166  & 4.479  & 5.982  & 2.400  & 8.205   & 9.842 \\
ftv35    & 3.216  & 5.675  & 6.885  & 3.033  & 10.390  & 14.332 \\
ftv38    & 2.191  & 6.913  & 9.661  & 2.919  & 12.618  & 16.558 \\
ftv44    & 3.386  & 7.231  & 11.101 & 3.646  & 13.903  & 16.910 \\
ft53     & 5.982  & 12.610 & 17.852 & 6.134  & 23.843  & 30.537 \\
ftv70    & 23.570 & 14.198 & 43.772 & 29.104 & 25.878  & 61.456 \\
kro124p  & 54.231 & 53.771 & 157.813 & 65.509 & 102.540 & 204.485 \\
\bottomrule
\end{tabular}
\end{table}
 
For symmetric instances (Table \ref{tab:cycap-time-symmetric}), running times of all three embedded variants are quite similar across the board. Recall that \textsc{Cycap-C}, when run alone, scales better than \textsc{Cycap-F} or \textsc{Cycap-M}. Due to the higher gap closure for \textsc{Cycap-C} in $k+C+k$ runs compared to $k+C$ runs, the second $k$-opt run in that setting goes through more iterations due to the room for improvement, and thus takes longer. In pure running time numbers, \textsc{Cycap-C} `loses' some of its overall speed advantage.

For directed instances (Table \ref{tab:cycap-time-directed}), \textsc{Cycap-C} is much faster than \textsc{Cycap-F}, particularly on large problems, where differences can reach an order of magnitude. Recall that, for these instances, \textsc{Cycap-C} is significantly faster than \textsc{Cycap-F} when run by itself. Additionally, \textsc{Cycap-C} already has an excellent gap closure in the $k+C$ setting, which means that a second $k$-opt run typically concludes comparatively quickly. In combination, this leads to the best performance among variants. The efficiency, combined with strong improvement capability as demonstrated by success rate and gap closure, makes \textsc{Cycap-C} competitive for large directed problems.

\section{Final Remarks}\label{sec:conclusion}

We proposed \textsc{Cycap}, a ``cycle cancel and patch'' heuristic for the directed (and symmetric) TSP that leverages a special bipartite encoding of the residual network for a tour to detect and cancel negative tour-alternating cycles or circulations. The key structural fact is that tour–alternating cancels preserve local balance and return either a tour or a small collection of subtours (possibly with some isolated vertices); in particular, any simple cycle whose cancel yields a new tour must be tour–alternating. Via a so-called separated graph, detecting such cycles or circulations reduces to standard negative-cycle or minimum-circulation routines. When cancels produce subtours, a lightweight patching step quickly reconnects the flow into a single tour.

In computational experiments across benchmark instances, \textsc{Cycap} exhibits a practially viable combination of success rate, gap closure, and running time, particularly so for directed instances and when followed by a second $k$-opt refinement. Among the implementation variants, \textsc{Cycap-C} (based on the search for an improving tour-alternating circulation via an LP on the separated graph) consistently performed best, with high success rates and large gap closure at practical running times. On symmetric instances, the other variants  \textsc{Cycap-F} and \textsc{Cycap-M} performed similarly well to \textsc{Cycap-C} in terms of success rate and gap closure, but at slower running times. In practice, some simple usage rules emerge: By design, \textsc{Cycap} performs better for directed TSP instances than for symmetric instances; circulation-based updates generally perform best; and the computational effort for a second $k$-opt run is worth it.

There are some interesting directions of future work. In particular, we are interested in understanding (and possibly exploiting) the fact that cancels typically produce only few subtours. In our practical experiments, the median number of subtours after a cancel is small (often $2$–$4$ and rarely above single digits) across all instances and variants. As the patching routine is a greedy part of \textsc{Cycap} that might increase cost, the fact that only few subtours are created and thus only few patches are necessary is one of the reasons for the good performance of \textsc{Cycap}. As expected, success rates are highest when the count is low. Some experiments on this behaviour can be found in \cite{z-25}. 

Clearly, the number of subtours created is related to the length of tour-alternating cycles. We observed that short tour–alternating cycles dominate in practice (lengths $4,6,8,10$). In \cite{z-25}, some structural observations partially explain this behavior: A simple tour-alternating cycle of length $2k$ cannot create more than $k$ subtours; in fact, it can only be exactly $k$ or strictly fewer than $k-1$ subtours. Thus, short cycles translate immediately into few subtours. A permutation viewpoint allows a characterization of the possible number of subtours for short tour-alternating cycles. For example, if the length is divisible by four, a cancel necessarily produces subtours \cite{z-25}. We are interested in a complete characterization of this behavior for tour-alternating cycles and circulations of $2k$ arcs. Algorithmically, the prevalence and benefits of short tour-alternating cycles suggest that it would be promising to design a combinatorial algorithm to directly search for tour-alternating cycles of bounded length, without making use of cycle-detection or linear programming subroutines.

Beyond a focused interest in short cycles, two practical follow-ups are natural. First, it could be valuable to implement and compare alternative cycle/circulation detection routines on the separated graph - beyond the current Floyd–Warshall, minimum–mean, and circulation LP variants - to include other standard options. Second, our proof-of-concept implementation clearly can be refined for larger instances. A particularly useful step could be the adoption of memory–friendly data structures for the separated graph.

\paragraph*{\bf Acknowledgments} This work was supported by the Air Force Office of Scientific Research under award number FA9550-24-1-0240.

\bibliographystyle{apalike}

\end{document}